\documentclass[11pt,a4paper,reqno]{amsart}

\usepackage[T1]{fontenc}

\usepackage[utf8]{inputenc}

\usepackage{xcolor}

\usepackage{scalerel,stackengine}
\stackMath
\newcommand\hathat[1]{%
\savestack{\tmpbox}{\stretchto{%
  \scaleto{%
    \scalerel*[\widthof{\ensuremath{#1}}]{\kern-.6pt\bigwedge\kern-.6pt}%
    {\rule[-\textheight/2]{1ex}{\textheight}}
  }{\textheight}%
}{0.5ex}}%
\stackon[1pt]{#1}{\tmpbox}%
}
\usepackage{amsmath,amsthm,amsopn,amsfonts,amssymb,color}
\usepackage{cancel}
 \usepackage{tikz}
\usetikzlibrary{intersections,calc}
\usepackage{graphicx}
\usepackage{float}
\usepackage{amsmath}
\usepackage{fancyhdr}
\usepackage{epstopdf}
\usepackage{amsfonts}
\usepackage{pgfplots}
\pgfplotsset{compat=newest}
\usetikzlibrary{intersections, pgfplots.fillbetween}
\pgfdeclarelayer{pre main}

\usepackage[hidelinks,colorlinks=false,linkcolor=black,urlcolor=black,citecolor=black,pdftitle={Recovering functions via doubly homogeneous nonlocal gradients}, pdfauthor={S. Buccheri and A.C. Ponce}]{hyperref}

\usepackage[abbrev,backrefs]{amsrefs}
\usepackage{mathtools}
\usepackage{esint}

\usepackage{booktabs}
\usepackage[USenglish]{babel}
\usepackage{cancel}
\usepackage{times}
\usepackage{mathrsfs}
\usepackage{multirow,multicol}
\usepackage{amsfonts}
\usepackage{amsthm}
\usepackage{amsmath}

\DeclareFontFamily{U}{mathx}{}
\DeclareFontShape{U}{mathx}{m}{n}{<-> mathx10}{}
\DeclareSymbolFont{mathx}{U}{mathx}{m}{n}
\DeclareMathAccent{\widehat}{0}{mathx}{"70}
\DeclareMathAccent{\widecheck}{0}{mathx}{"71}

\usepackage{enumerate}
\usepackage{constants}

 \usepackage{color}
\usepackage{todonotes}

\theoremstyle{plain}
\newtheorem{proposition}{Proposition}
\newtheorem{lemma}[proposition]{Lemma}
\newtheorem{theorem}[proposition]{Theorem}
\newtheorem{corollary}[proposition]{Corollary}
\newtheorem{definition}[proposition]{Definition}

\theoremstyle{definition}

\theoremstyle{remark}
\newtheorem{openproblem}[proposition]{Open problem}

\newtheorem*{Claim}{Claim}

\usepackage{etoolbox}
\theoremstyle{definition}
\newtheorem{example}[proposition]{Example}
\AtEndEnvironment{example}{\quad\null\hfill\qedsymbol}

\newcommand{\abs}[1]{\mathopen\lvert#1\mathclose\rvert}

\newcommand{\biggabs}[1]{\biggl\lvert#1\biggr\rvert}
\newcommand{\norm}[1]{\mathopen\lVert#1\mathclose\rVert}

\newcommand{\N}{{\mathbb N}}
\newcommand{\R}{{\mathbb R}}
\newcommand{\CC}{{\mathbb C}}
\newcommand{\dif}{\,\mathrm{d}}
\newcommand{\difq}{\mathrm{d}} 

\DeclareMathOperator{\Div}{div}

\newcommand{\Fourier}{\mathcal{F}}
\newcommand{\FourierF}[1]{[\mathcal{F}#1]}
\newcommand{\DistribF}[1]{[#1]}
\newcommand{\DistribS}[1]{S_{#1}}
\newcommand{\QDeriv}{\mathcal{G}}
\newcommand{\QPartial}[1]{\mathcal{G}_{#1}}
\newcommand{\Average}[1]{\overline{#1}}

\newcommand{\loc}{_\mathrm{loc}}
\newcommand{\e}{\mathrm{e}}

\usepackage{mathrsfs} 

\begin{document}

\title{Recovering functions via doubly homogeneous nonlocal gradients}

\author{Stefano Buccheri}
\address{
Stefano Buccheri\hfill\break\indent
Università degli Studi di Napoli Federico II \hfill\break\indent
Dipartimento di Matematica e Applicazioni R. Caccioppoli\hfill\break\indent
Via Vicinale dell'Infermeria, 58\hfill\break\indent
80125 Napoli, Italy}

\author{Augusto C. Ponce}
\address{
Augusto C. Ponce\hfill\break\indent
 Université catholique de Louvain\hfill\break\indent
 Institut de Recherche en Mathématique et Physique\hfill\break\indent
 Chemin du cyclotron 2, L7.01.02\hfill\break\indent
1348 Louvain-la-Neuve, Belgium}


\dedicatory{
To the memory of Haim Brezis, with deep admiration and gratitude.}


\begin{abstract}
We investigate a class of nonlocal gradients featuring distinct homogeneities at zero and infinity. 
We establish a representation formula for such doubly homogeneous operators and derive associated Sobolev-type inequalities.
We also propose open questions linked to our results, suggesting directions for future research inspired by the work of Haim Brezis.
\end{abstract}

\maketitle

\section{Introduction}
We study a class of \emph{nonlocal gradients} whose canonical example associated to a function $u \in C^\infty_c(\mathbb{R}^d)$ in any dimension $d \geq 1$ is the Riesz fractional gradient of order $0 < s < 1$\,:
\begin{equation*}
\nabla^s u(x) = \int_{\mathbb{R}^d} (u(x) - u(y)) \frac{x-y}{|x-y|^{d+s+1}} \dif y,
\end{equation*}
which reduces to the Riesz transform for $s = 0$.
When \(s = 1\) one should not expect to obtain the classical gradient since as observed by Bourgain, Brezis and Mironescu~\cite{BBM-2001} one has
\[
\int_{\mathbb{R}^d}\int_{\mathbb{R}^d} \frac{\abs{u(x) - u(y)}}{|x-y|^{d+1}} \dif x \dif y = \infty,
\]
unless \(u\) is a constant. This is due to a defect in the definition of \(\nabla^s\) that does not take into consideration its behavior as \(s \to 1\).
In fact, the \emph{correct} operator that allows one to recover \(\nabla\) should be instead \((1-s)\nabla^s\), as one verifies that
\begin{equation}
\label{eq-318}
(1 - s)\nabla^s u \to \frac{\sigma_d}{d} \nabla u
\quad \text{when \(s \to 1\),}
\end{equation}
both uniformly and in \(L^1(\R^d)\), where \(\sigma_d\) denotes the area of the unit sphere \(\partial B_1\) in \(\R^d\).

The fractional gradient $\nabla^s$ introduced in \cite{H} within singular integral theory has attracted significant attention in both applied and theoretical mathematics. For example, in nonlocal continuum mechanics \cite{MMW} and image processing \cite{Gilboa}, see also \cite{Sil} for further references, but also in theoretical developments including Bourgain-Brezis-Mironescu-type results \cite{MS}, variational nonlocal elliptic operators \cites{SS,SSbis}, and distributional fractional calculus \cites{CS,BCCS,Sil}.

Generalizations of $\nabla^s$ have also emerged in elasticity and peridynamics \cites{BCMC,BMCS,ACFS}, where the singular kernel is modified to define broader classes of nonlocal gradients.
In this paper, we advance this direction by focusing on \emph{doubly homogeneous gradients} --- a flexible class of nonlocal operators --- and establishing results for their associated function spaces and inversion theory.

The nonlocal gradients we consider are defined as follows:

\begin{definition}
\label{def:G-gradient}
Let \( g \in C^\infty(\mathbb{R}^d \setminus \{0\}) \) satisfy the integrability condition
\begin{equation}
\label{eq:gradient_condition}
\int_{B_1} |x| |\nabla g(x)| \dif x 
+ \int_{\mathbb{R}^d \setminus B_1} |\nabla g(x)| \dif x < \infty.
\end{equation}
The \emph{$\mathcal{G}$-fractional gradient} of a function \( u \in C_c^\infty(\mathbb{R}^d) \) is defined pointwise as
\begin{equation}
\label{eqFractionalGradient}
\mathcal{G}u(x) = -\int_{\mathbb{R}^d} (u(x) - u(y)) \nabla g(x - y) \dif y.
\end{equation}
\end{definition}

The choice \( g(x) = |x|^{-(d - 1 + s)} \) with \( 0 < s < 1 \) recovers, up to a multiplicative constant, the Riesz fractional gradient \( \nabla^s u \). For this case, there exists a vector field \( V_s(x) = A x/|x|^{d - s + 1} \), for some constant \(A\) depending on \(s\) and \(d\), such that one has the following inversion formula~ \cite{Ponce}*{Proposition~15.8}:
\begin{equation}
\label{eq:fundamental_theorem}
u = V_s * \nabla^s u
\quad \text{in \(\R^d\).}
\end{equation}
This serves as a fractional counterpart to the fundamental theorem of calculus.

The identity \eqref{eq:fundamental_theorem} demonstrates that smooth, compactly supported functions \( u \) can be fully recovered from their Riesz fractional gradients. Such representation formulas are powerful tools for proving embedding theorems via convolution estimates. For instance, Sobolev- and Hardy-type inequalities for \( \|\nabla^s u\|_{L^p} \) with \( p > 1 \) were established by Shieh and  Spector in \cite{SS}.
The endpoint case \( p = 1 \) presents additional subtleties: as shown by Schikorra, Spector and Van~Schaftingen \cite{SSVS}, the inequality
    \begin{equation}
    \label{eq:L1_inequality}
    \|u\|_{L^{\frac{d}{d-s}}(\mathbb{R}^d)} \leq C \|\nabla^s u\|_{L^1(\mathbb{R}^d)}
    \end{equation}
    holds for \( d \geq 2 \), with a proof that relies on duality and the curl-free property of the fractional gradient.
    
    In Section~\ref{sectionNonlocalSobolevinequalities} we provide a streamlined proof of \eqref{eq:L1_inequality} that follows the strategy of \cite{SSVS}, based on an idea from \cite{VS}, while eliminating the need for \(L^p\)-bounds on the Riesz transform.
    A careful inspection of that proof gives a constant in \eqref{eq:L1_inequality} that is compatible with the limit behavior of \((1-s)\nabla^s\) as \(s \to 1\) given by \eqref{eq-318}.
    More precisely,

    \begin{theorem}
    \label{propositionEstimateRieszFractionalGradient}
    Let \(d \ge 2\).
    There exists a constant \(\widetilde{C} > 0\) such that, for every \(u \in C_c^\infty(\R^d)\) and every \(1/2 \le s < 1\),
    \begin{equation}
    \label{eq:L1_inequality-Bis}
    \|u\|_{L^{\frac{d}{d-s}}(\mathbb{R}^d)} \leq \widetilde{C} (1-s) \|\nabla^s u\|_{L^1(\mathbb{R}^d)}.
    \end{equation}
    \end{theorem}

    As a consequence, one recovers the classical Sobolev inequality as \(s \to 1\) in the spirit of analogous results of Bourgain, Brezis and Mironescu~\cites{BBM-2002, BBM-2004} concerning the behavior of the constant with respect to the fractional parameter.
    Note that 
    \[
    \|\nabla^s u\|_{L^1(\mathbb{R}^d)}
    \le \int_{\R^d}\int_{\R^d} \frac{\abs{u(x) - u(y)}}{\abs{x - y}^{d + s}} \dif x \dif y,
    \]
    where the right-hand side is the Gagliardo seminorm associated to the fractional Sobolev space \(W^{s, 1}(\R^d)\).
    Hence, the estimate \eqref{eq:L1_inequality-Bis} gives a stronger version of the fractional Sobolev inequality that is stable in the limit as \(s \to 1\) and had been obtained using various methods \cites{Kolyada-Lerner:2005,Almgren-Lieb:1989,Frank-Seiringer:2008,BBM-2002,MSh-2002,PonceSpector:2020}.
    
In light of the representation formula \eqref{eq:fundamental_theorem}, it is therefore natural to investigate what conditions on the nonlocal operator $\mathcal{G}$ permit a representation formula, and what analytical insights such a formula might provide. We focus on kernels $g$ exhibiting \emph{double homogeneity}, with distinct scaling behaviors near the origin and at infinity.  
We do not focus on the cases of dimensions \(d = 1\) and \(2\) to avoid technical details specific to these dimensions.
Three representative examples illustrate our framework:  
\begin{enumerate}
    \item \emph{Localized fractional gradient}: For $0 < r < R$ and $0 < s < 1$,  
\begin{equation}\label{eqGLocal}  
g(x) = \begin{cases}  
1/|x|^{d-1+s} & \text{if } |x| \leq r, \\  
0 & \text{if } |x| \geq R,  
\end{cases}  
\end{equation}  
which restricts the fractional derivative to a bounded region.  
    \item \emph{Integrable tail}: For $\alpha > d$ and $a, b > 0$,  
\begin{equation}\label{eqGIntermediate}  
g(x) = \begin{cases}  
a/|x|^{d-1+s} & \text{if } |x| \leq r, \\  
b/|x|^{\alpha} & \text{if } |x| \geq R,  
\end{cases}  
\end{equation}  
ensuring integrability at infinity while preserving fractional behavior locally.  
    \item \emph{Two-scale fractional kernel}: For $0 < s, t < 1$,  
\begin{equation}\label{eqGTwo}  
g(x) = \begin{cases}  
a/|x|^{d-1+s} & \text{if } |x| \leq r, \\  
b/|x|^{d-1+t} & \text{if } |x| \geq R,  
\end{cases}  
\end{equation}  
interpolating between different fractional regimes.  
\end{enumerate}
We shall also suppose that
\begin{equation}\label{assurepositivityofFourier}
\left\{
\begin{aligned}
& \text{\( g \in C^\infty(\mathbb{R}^d \setminus \{0\}) \) is radial and}\\
& \text{\(\rho \in(0,\infty) \longmapsto \rho^{d-1}g(\rho x)\) 
\ is non-increasing and convex for any $x\neq 0$.}
\end{aligned}
\right.
\end{equation}
These assumptions ensure that the Fourier transform of \(g\) is well-defined and positive, which is a natural request within our approach, see \eqref{introconvolution} and also Sections~\ref{sectionfouriertransformradialhomogeneous} and~\ref{sectionPositive} below. 

For such kernels, we establish a general representation formula:  
\begin{theorem}\label{introrepfor}  
Let $d\ge3$ and suppose that $g$ given by \eqref{eqGLocal}, \eqref{eqGIntermediate}, or \eqref{eqGTwo} satisfies \eqref{assurepositivityofFourier}. 
Then, there exists a smooth vector field $V \colon \mathbb{R}^d \setminus \{0\} \to \mathbb{R}^d$ such that, for every $u \in C_c^\infty(\mathbb{R}^d)$,  
\begin{equation}
\label{eqRepresentationFormula}
u = V * \mathcal{G}u \quad \text{in } \mathbb{R}^d.  
\end{equation}
\end{theorem}

Under the assumption \eqref{eqGLocal} such a representation formula has been obtained by Bellido, Cueto and Moral-Corral~\cite{BCMC}, see also \cite{BMCS}.

The proof of \eqref{introconvolution} also provides precise estimates on the behavior of $V$ and its derivatives. For the kernels \eqref{eqGLocal} and \eqref{eqGIntermediate}, we obtain for every multi-index $\nu$:
\begin{equation} \label{introV1}
|\partial^{\nu}V(x)| \leq 
\begin{cases}
{C}/{|x|^{d - s + |\nu|}} & \text{if } |x| \leq 1, \\
{C}/{|x|^{d - 1 + |\nu|}} & \text{if } |x| \geq 1.
\end{cases}
\end{equation}
This shows that both compactly supported kernels and those with integrable tails induce identical decay for $V$ at infinity. It may be of some interest to notice that at infinity the same behaviour of the classical local representation formula is recovered. 
In contrast, for the doubly homogeneous kernel \eqref{eqGTwo} the estimates reflect its distinct scaling regimes:
\begin{equation} \label{introV2}
|\partial^{\nu}V(x)| \leq 
\begin{cases}
{C}/{|x|^{d - s + |\nu|}} & \text{if } |x| \leq 1, \\
{C}/{|x|^{d - t + |\nu|}} & \text{if } |x| \geq 1.
\end{cases}
\end{equation}

The core idea to prove Theorem~\ref{introrepfor} is to find $\omega \in C^\infty(\mathbb{R}^d)$ satisfying:
\begin{equation}\label{introconvolution}
\omega * g(z) = \frac{1}{|z|^{d-2}} \quad \text{for all } z \neq 0.
\end{equation}
Most of this paper is devoted to the solution of this convolution equation. 
This approach requires a careful analysis of the Fourier transform of the homogeneous function ${1}/{|x|^{d-\alpha}}$ which does not belong to the usual \(L^1\) or \(L^2\) settings (Sections~\ref{sectioncharacterizationradialhomogeneous} and~\ref{sectionfouriertransformradialhomogeneous}), the computation of the Fourier transform of \(g\) (Section~\ref{sectionfouriertransformdoublyhomogeneous}), and the inversion of the transformed solution (Section~\ref{sectionexistenceofthekernel}). 

Defining
\begin{equation*}
V = -\frac{1}{\sigma_d (d-2)} \nabla \omega
\end{equation*}
and applying Proposition~\ref{propositionRepresentationGradient} then yields:
\[
V * \mathcal{G}u = V * g * \nabla u = \left( \frac{1}{\sigma_d} \frac{z}{|z|^d} \right) * \nabla u = u.
\]

In the proof of Theorem~\ref{introrepfor}, we strongly rely on the explicit formulas of \(g\) near \(0\) and infinity. 
It would be interesting to have a weaker assumption that relies mostly on the behavior of \(g\) in these regions.
A common roof that could collect the assumptions \eqref{eqGLocal} and \eqref{eqGIntermediate} would be the following:

\begin{openproblem}
    Assume that \(g \in C^\infty( \R^d \setminus \{0\} )\) is a radial function such that
    \begin{equation}
    \label{eqAssumption1}
    \lim_{x \to 0}{\abs{x}^{d - 1 + s} g(x)} \in (0, \infty)
    \end{equation}
    and
    \begin{equation}
    \label{eqAssumption2}
    \abs{g(x)} \le C/\abs{x}^\alpha
    \quad \text{for \(\abs{x} \ge R\),}
    \end{equation}
    where \(0 < s < 1\) and \(\alpha > d\). 
    Does there exist a smooth vector field $V \colon \mathbb{R}^d \setminus \{0\} \to \mathbb{R}^d$ such that the representation formula \eqref{eqRepresentationFormula} holds for every $u \in C_c^\infty(\mathbb{R}^d)$\,?
\end{openproblem}

An immediate consequence of \eqref{eqRepresentationFormula} is that if \(u \in C_c^\infty(\R^d)\) is such that \(\mathcal{G}u = 0\) in \(\R^d\), then \(u = 0\).
In the spirit of H.~Brezis' work~\cite{Brezis} that characterizes measurable functions that are constant almost everywhere
based on the condition
\[
\int_{\mathbb{R}^d}\int_{\mathbb{R}^d} \frac{\abs{u(x) - u(y)}}{|x-y|^{d+1}} \dif x \dif y < \infty,
\]
see also \cite{BrezisMironescu}*{Section~6.2} for alternative elementary proofs, one could extend the fractional derivative \(\mathcal{G}u\) and investigate an analogous question to functions that are not necessarily smooth.
Indeed, if \(u \colon \R^d \to \R\) is a measurable function that satisfies
\begin{equation}
\label{eqConditionIntegrabilityG}
\int_{\mathbb{R}^d}\int_{\mathbb{R}^d} \abs{u(x) - u(y)} \abs{\nabla g(x - y)} \dif x \dif y < \infty,
\end{equation}
then, by Fubini's theorem, the formula \eqref{eqFractionalGradient} for \(\mathcal{G}u(x)\) is well-defined for almost every \(x \in \R^d\).

\begin{openproblem}
    Assume that  \(g \in C^\infty( \R^d \setminus \{0\} )\) is radial and satisfies \eqref{eqAssumption1} and \eqref{eqAssumption2}.
    If \(u\) is a measurable function in \(\R^d\) that verifies \eqref{eqConditionIntegrabilityG} and \(\mathcal{G} u = 0\) almost everywhere in \(\R^d\), is \(u\) a constant in \(\R^d\)\,?
\end{openproblem}

Note that for a given function \(g\) the representation formula may fail, but still the fractional gradient could characterize the constant functions. A further challenging question would be:

\begin{openproblem}
Identify functions  \(g \in C^\infty( \R^d \setminus \{0\} )\) for which the condition \(\mathcal{G}u = 0\) characterizes all constant measurable functions \(u\) that satisfy \eqref{eqConditionIntegrabilityG}.    
For example, is this the case if the Fourier transform of \(g\) is well-defined and almost everywhere nonzero?
\end{openproblem}

The next results concern fractional Sobolev inequalities in our doubly homogeneous setting. Clearly, one does not expect to have an estimate for $u$ in a single Lebesgue space, but rather in a suitable space that accommodates the difference between homogeneity rates. The sum of Lebesgue spaces represents a valid alternative to this issue, as we show in the next couple of results.

\begin{theorem}\label{introsobemb1}
        Take $d\ge3$ and $1\le p<d/s$ with $0 < s < 1$.
        Assume that \(g\) is given by \eqref{eqGLocal} or \eqref{eqGIntermediate} and satisfies \eqref{assurepositivityofFourier}. 
        Then, for every \(u \in C_{c}^{\infty}(\R^{d})\),{}
\[{}
\norm{u}_{L^{\frac{pd}{d - sp}} + L^{\frac{pd}{d - p}}}
\le C \norm{\QDeriv{u}}_{L^{p}(\R^d)}.
\]
\end{theorem}

Unless otherwise stated, $C>0$ is an absolute constant that may depend on the parameters $d$, $s$, $t$, $p$ and on the function $g$ but that does not depend on the function $u$.
Let us recall that $u$ belongs to $(L^m+L^q)(\R^d)$ whenever there exist $u_1\in L^m(\R^d)$ and $u_2\in L^q(\R^d)$ so that $u=u_1+u_2$, in which case \(\|u\|_{L^m+L^q}\) is the infimum of the sum
\[
\|u_1\|_{L^m(\R^d)}+\|u_2\|_{L^q(\R^d)}
\]
over all such decompositions of \(u\).

In the same spirit as Theorem~\ref{introsobemb1} we have:

\begin{theorem}\label{introsobemb2}
         Take $d\ge3$ and $1 \le p < \min{\{d/s,d/t\}}$ with $0 < s,t <1$.
         Assume that \(g\) is given by \eqref{eqGTwo} and satisfies \eqref{assurepositivityofFourier}. 
         Then, for every \(u \in C_{c}^{\infty}(\R^{d})\),{}
\[{}
\norm{u}_{L^{\frac{pd}{d - sp}} + L^{\frac{pd}{d - tp}}}
\le C \norm{\QDeriv{u}}_{L^{p}(\R^d)}.
\]
\end{theorem}

The first step in the proof of Theorem~\ref{introsobemb2} is to write the vector field \(V\) as $V=V_1+V_2$ so that $V_1$ behaves as $1/|x|^{d-s}$ and $V_2$ as $1/|x|^{d-t}$.
Therefore, the composition formula provided by Theorem~\ref{introrepfor} becomes
\[
u= V_1*\mathcal{G}u+V_2*\mathcal{G}u.
\]
The second step is to treat the two contributions on the right hand side above. 
While for $p>1$ it is enough to rely on classical estimates for the Riesz potential, the borderline case $p=1$ requires an adaptation of the more delicate proof of the already mentioned \cite{SSVS}*{Theorem~A}, see Section~\ref{sectionNonlocalSobolevinequalities} below. 

Actually, for $p>1$ we may improve Theorems~\ref{introsobemb1} and \ref{introsobemb2} by both refining the embedding on the Lorentz scale and having a better description of the influence of the different homogeneities of $g$.

\begin{theorem} \label{introsobemb3}
        Take $d\ge3$ and $1 < p < d/s$ with $0 < s < 1$.
         Assume that \(g\) is given by \eqref{eqGLocal} or \eqref{eqGIntermediate} and satisfies \eqref{assurepositivityofFourier}. 
        Then, for every \(u \in C_{c}^{\infty}(\R^{d})\), there exists $k\ge 0$ depending on \(u\) such that
\[
\|G_k(u)\|_{L^{\frac{pd}{d-sp},p}(\mathbb{R}^d)}+\|T_k(u)\|_{L^{\frac{pd}{d-p},p}(\mathbb{R}^d)} \le C \norm{\QDeriv{u}}_{L^{p}(\mathbb{R}^d)}.
\]
\end{theorem}

Here, 
\[ 
u = G_k(u) + T_k(u)
\]
where the truncations \(G_k\) and \(T_k\) are explicitly defined for every \(\tau \in \R\) as
\[
T_k(\tau) = \max{\{-k,\min{\{\tau,k\}}\}}
\quad \text{and} \quad 
G_k(\tau) = \tau- T_k(\tau).
\]
In the same spirit, concerning the two-scale fractional kernel, the following holds:

\begin{theorem}\label{introsobemb4}
        Take $d\ge3$ and $1 < p < \min{\{d/s,d/t\}}$ with $0 < s,t < 1$.
         Assume that \(g\) is given by \eqref{eqGTwo} and satisfies \eqref{assurepositivityofFourier}. 
	Then, for every \(u \in C_{c}^{\infty}(\R^{d})\), there exists $k \ge 0$ depending on \(u\) such that
\[
\|G_k(u)\|_{L^{\frac{pd}{d-sp},p}(\mathbb{R}^d)}+\|T_k(u)\|_{L^{\frac{pd}{d-tp},p}(\mathbb{R}^d)} 
\le C \norm{\QDeriv{u}}_{L^{p}(\mathbb{R}^d)}.
\]
\end{theorem}

We observe in the previous statement that the estimate of $G_k(u)$ in the Lorentz spaces $L^{\frac{pd}{d-sp},p}(\mathbb{R}^d)$ tells us that the local integrability of $u$ is related to the singularity at zero of the kernel $g$. Similarly, the estimate of $T_k(u)$ expresses the integrability at infinity of $u$ that depends on the tail of $g$.

The level at which we slice the function $u$ in our proof is taken as $k=\norm{\QDeriv{u}}_{L^{p}(\mathbb{R}^d)}$. 
Let us point out that the level of truncation $k$ cannot be chosen independently of \(u\). 
Indeed, if this was the case, taking the limit as $k$ goes to $0$ or infinity, one would obtain an estimate where the influence of one of the two homogeneities of $g$ is lost.

For the Riesz fractional gradient, namely when $g(x)=c/|x|^{d-1+s}$, Spector~\cite{Spector} showed that Theorem~\ref{introsobemb4} is valid for $p=1$ and $s=t$. 
This leads to the following question:

\begin{openproblem}
	  Given \(g\) of the form \eqref{eqGTwo}, is it true that for every $u\in C^{\infty}_c(\mathbb{R}^d)$ there exists a truncation level $k=k(u)$ such that
    \[
\|G_k(u)\|_{L^{\frac{d}{d-s},1}(\mathbb{R}^d)}+\|T_k(u)\|_{L^{\frac{d}{d-t},1}(\mathbb{R}^d)} 
\le C \norm{\QDeriv{u}}_{L^{1}(\mathbb{R}^d)} ?
\]
\end{openproblem}

\section{Relation between \texorpdfstring{$\QDeriv u$}{Gu} and \texorpdfstring{$\nabla u$}{grad u}}

We begin by showing that the fractional gradient can be seen as a convolution of the classical gradient:

\begin{proposition}
\label{propositionRepresentationGradient}
If \(g \in C^\infty(\R^d \setminus \{0\})\) satisfies \eqref{eq:gradient_condition}, then \(g \in L^1\loc(\R^d)\) and, for every \(u \in C_c^\infty(\R^d)\), we have
\[
\QDeriv u
= g * \nabla u
\quad \text{in \(\R^d\).}
\]
\end{proposition}

This identity, implicit for example in \cite{MS,DGLZ}, was highlighted by Brezis and Mironescu~\cite{BrezisMironescu:2023} as a key ingredient behind nonlocal approximations of the gradient in the spirit of \eqref{eq-318}.
The composition in the right-hand side is well-defined since \(\nabla u\) is a bounded function with compact support and \(g\) is locally summable.
That the latter holds follows from

\begin{lemma}
\label{lemmaQEstimate}
If \(g \in C^\infty(\R^d \setminus \{0\})\)  satisfies \eqref{eq:gradient_condition}, then 
\[
\int_{B_1} |g| + \int_{\R^d \setminus B_1} |g|^\frac{d}{d-1} < \infty.
\]
In particular, \(g \in (L^1 + L^\frac{d}{d-1})(\R^{d})\).
\end{lemma}

\resetconstant
\begin{proof}[Proof of Lemma~\ref{lemmaQEstimate}]
Take a smooth function \(\psi \colon  \R^d \to \R\) with \(\psi = 1\) in \(\R^d \setminus B_1\) and \(\psi = 0\) in \(B_{1/2}\). 
Since \(\nabla(g\psi) \in L^1(\R^d)\), by the Sobolev inequality we have \(g\psi \in L^\frac{d}{d-1}(\R^d)\) and then
\[
\biggl(\int_{\R^d \setminus B_1} |g|^\frac{d}{d-1}\biggr)^\frac{d-1}{d}
 \le \norm{g\psi}_{L^\frac{d}{d - 1}}
 \le \C \norm{\nabla(g\psi)}_{L^1}
< \infty.
\]

Next, by the  Fundamental theorem of Calculus, for every \(0 < r < 1\) and every \(y \in \partial B_1\) we have
\[
\abs{g(ry)}
= \biggabs{g(y) - \int_r^1 \frac{\difq}{\difq t} g(ty) \dif t}
\le \abs{g(y)} + \int_r^1 \abs{\nabla g(t y)} \dif t.
\]
We then multiply both sides by \(r^{d - 1}\) and integrate with respect to \(r\).
Applying Tonelli's theorem,
\[
\int_0^1 \abs{g(ry)} r^{d - 1}  \dif r
\le \frac{1}{d} \abs{g(y)} + \int_0^1 \biggl(\int_0^t \abs{\nabla g(t y)} r^{d - 1} \dif r \biggr) \dif t
= \frac{1}{d} \abs{g(y)} + \frac{1}{d} \int_0^1  \abs{\nabla g(t y)} t^{d} \dif t.
\]
Integrating with respect to \(y\) and applying the integration formula in polar coordinates, we obtain
\[
\int_{B_1} \abs{g}
= \int_{\partial B_1} \biggl(  \int_0^1 \abs{g(ry)} r^{d - 1}  \dif r \biggr) \dif\sigma(y)
\le \frac{1}{d} \int_{\partial B_1} \abs{g} \dif\sigma + \frac{1}{d} \int_{B_1}  \abs{\nabla g(x)} \abs{x}\dif x,
\]
which gives the conclusion.
\end{proof}

\begin{proof}[Proof of Proposition~\ref{propositionRepresentationGradient}]
Since \(u\) is a bounded smooth function and \(g\) satisfies \eqref{eq:gradient_condition}, we have 
\[
\QDeriv u(x)
= \lim_{\substack{r \to 0\\R \to \infty}}{\int_{B_R(x) \setminus B_r(x)} (u(x) - u(y)) \nabla_y g(x - y) \dif y}.
\]
Since \(g\) is locally summable and \(\nabla u\) has compact support, we also have
\[
g * \nabla u(x)
= - \lim_{\substack{r \to 0\\R \to \infty}}{\int_{B_R(x) \setminus B_r(x)} \nabla_{y} ( u(x) - u(y) ) g(x - y) \dif y}.
\]
For every \(0 < r < R\) and \(j \in \{1, \dots, d\}\), by the Divergence theorem we have
\begin{multline}
\label{eqRepresentation-273}
\int_{B_R(x) \setminus B_r(x)} \nabla_{y} ( u(x) - u(y)) g(x - y) \dif y - \int_{B_R(x) \setminus B_r(x)} (u(x) - u(y)) \nabla_y g(x - y) \dif y\\
\begin{aligned}
& = \int_{B_R(x) \setminus B_r(x)} \nabla_{y} \bigl[ ( u(x) - u(y) ) g(x - y)\bigr] \dif y\\
& = \int_{\partial B_R(x)} ( u(x) - u(y) ) g(x - y) \nu(y) \dif\sigma(y)
- \int_{\partial B_r(x)} ( u(x) - u(y) ) g(x - y) \nu(y) \dif\sigma(y).
\end{aligned}
\end{multline}
We estimate
\[
\begin{split}
\biggl| \int_{\partial B_R(x)} ( u(x) - u(y) ) g(x - y) \nu(y) \dif \sigma(y) \biggr|
& \le 2 \norm{u}_{L^\infty} \int_{\partial B_R(x)} \abs{g(x - y)} \dif\sigma(y)\\
& = 2 \norm{u}_{L^\infty} \int_{\partial B_R} \abs{g} \dif\sigma.
\end{split}
\]
and
\[
\begin{split}
\biggl| \int_{\partial B_r(x)} ( u(x) - u(y) ) g(x - y) \nu(y) \dif y \biggr|
& \le \norm{\nabla u}_{L^\infty} \int_{\partial B_r(x)} \abs{x - y} \abs{g(x - y)} \dif \sigma(y)\\
& = \norm{\nabla u}_{L^\infty} \, r \int_{\partial B_r} \abs{g} \dif\sigma.
\end{split}
\]

We now show that there exist sequences of positive numbers \((r_j)_{j \in \N}\) and \((R_j)_{j \in \N}\) such that \(r_j \to 0\) and \(R_j \to \infty\) with
\[
\lim_{j \to \infty}{r_j \int_{\partial B_{r_j}} \abs{g} \dif\sigma}
= \lim_{j \to \infty}{\int_{\partial B_{R_j}} \abs{g} \dif\sigma}
= 0.
\]
To this end, we apply the integration formula in polar coordinates to get
\[
\int_0^1 t \biggl( \int_{\partial B_{t}} \abs{g} \dif\sigma \biggr) \frac{\dif t}{t}
= \int_{B_1} |g| < \infty,
\]
which gives the existence of the sequence \(r_j \to 0\) since \(\int_0^1 \difq t/t = \infty\).
For the other sequence, by Hölder's inequality we also have
\[
\int_1^\infty \biggl( \int_{\partial B_{t}} \abs{g} \dif\sigma \biggr)^{\frac{d}{d-1}} \frac{\dif t}{t} 
\le \Cl{cte-420} \int_1^\infty t \biggl( \int_{\partial B_{t}} \abs{g}^{\frac{d}{d-1}} \dif\sigma \biggr) \frac{\dif t}{t}
= \Cr{cte-420} \int_{\R^d \setminus B_1} |g|^{\frac{d}{d-1}} < \infty,
\]
Since \(\int_1^\infty \difq t/t = \infty\), there exists a sequence \(R_j \to \infty\) with the required property.
Taking \(R = R_j\) and \(r = r_j\) in \eqref{eqRepresentation-273} and letting \(j \to \infty\), we get
\[
- \nabla u * g(x)
+ \QDeriv u(x) = 0,
\]
from which the conclusion follows.
\end{proof}

\begin{corollary}\label{corollaryfromVtorepresentationfomula}
Let \(g \in C^\infty(\R^d \setminus \{0\})\) be a function that satisfies \eqref{eq:gradient_condition}.
If there exists \(V \in (L^{1} + L^{d})(\R^d)\) such that, for every \(z \in \R^d \setminus \{0\}\), 
\[
V * g (z)
= \frac{1}{\sigma_d}\frac{z}{|z|^d},
\]
where \(\sigma_d\) is the area surface of the sphere \(\partial B_1\),
then, for every \(u \in C_c^\infty(\R^d)\),
\[
u
= V * \QDeriv u
\quad \text{in \(\R^d\).}
\]
\end{corollary}

\begin{proof}
Since \(V \in (L^{1} + L^{d})(\R^d)\) and \(g \in (L^{1} + L^{\frac{d}{d-1}})(\R^d)\), we have \(\abs{V} * \abs{g} \in (L^{1} + L^{\infty})(\R^d)\).
By Proposition~\ref{propositionRepresentationGradient} and Fubini's theorem,
\[
V * \QDeriv u
= V * (\nabla u * g)
= \nabla u * (V * g).
\]
By the assumption on \(V * g\) and the classical representation formula involving the gradient, we then have for every \(x \in \R^d\),
\[
V * \QDeriv u (x)
= \frac{1}{\sigma_d} \int_{\R^d} \nabla u(x - y) \frac{y}{|y|^d} \dif y
= u(x).
\qedhere
\]
\end{proof}

We denote the \(j\)th component of \(\QDeriv u\) with respect to the canonical basis \(e_1, \dots, e_d\) by \(\QPartial{j} u\), so that
\[
\QDeriv u
= \sum_{j = 1}^d{\QPartial{j} u \, e_j}.
\]
It follows from the relation with the classical gradient that the fractional gradient is also curl free in the following sense:

\begin{corollary}
Let \(g \in C^\infty(\R^d \setminus \{0\})\) be a function that satisfies \eqref{eq:gradient_condition}.
If \(u \in C_c^\infty(\R^d)\), then
\[
\partial_l (\QPartial{j} u)
= \partial_j (\QPartial{l} u)
\quad \text{for every \(j, l \in \{1, \dots, d\}\).}
\]
\end{corollary}

\begin{proof}
By Proposition~\ref{propositionRepresentationGradient}, we have \(\QPartial{j} u = \partial_j u * g\).
Then, interchanging the order of integration and differentiation,
\[
\partial_l (\QPartial{j} u)
= \partial_l (\partial_j u * g)
= (\partial_l \partial_j u) * g.
\]
By smoothness of \(u\), we may exchange the order of the derivatives and obtain the identity in the statement.
\end{proof}

\section{Nonlocal Sobolev inequalities in Lebesgue spaces} \label{sectionNonlocalSobolevinequalities}

We present the strategy of the proof of the fractional Sobolev inequality from \cite{SSVS} in a way that makes easier to track the dependence of the constants involved.
Our aim is to justify the multiplicative factor \(1 - s\) in \eqref{eq:L1_inequality-Bis}.
The heart of the matter can be summarized in the next inequality:

\begin{proposition}
\label{propositionFractionalSobolev}
Let \(d \ge 2\) and \(0 < s < d\).{}
If \(v \in C^\infty(\R^d \setminus \{0\})\) is such that, for every \(z \in \R^d \setminus \{0\}\),
\begin{equation}
    \label{eqFractionalSobolevAssumption}
\abs{v(z)} + \abs{z} \abs{\nabla v(z)} 
\le \frac{C}{|z|^{d - s}}\text{,}
\end{equation}
then there exists a constant \(C'> 0\), depending on \(C\) and \(d/s\), such that
\[
\|v * F\|_{L^{\frac{d}{d - s}}}
\le C' \|F\|_{L^1}
\]
for each map \(F \in L^{1}(\R^d; \R^d)\) such that, for every \(l, j \in \{1, \dots, d\}\),
\begin{equation}
    \label{eqFractionalSobolevRotationalZero}
\partial_l F_j = \partial_j F_l \quad \text{in the sense of distributions in \(\R^{d}\).}
\end{equation}
\end{proposition}

The conclusion fails without the curl-free condition \eqref{eqFractionalSobolevRotationalZero}. 
This key-type of assumption --- which unlocks new elliptic estimates --- was first brought to light and beautifully explored by Bourgain and Brezis~\cites{BourgainBrezis-03,BourgainBrezis-07} in their foundational work on div-curl estimates.

We begin with the following standard estimate:

\begin{lemma}
\label{lemmaEstimateConvolutionV}
Let \(d \ge 1\), \(0 < s < d\) and \(\rho \in C_c^\infty(B_{1})\).
Then, for every \(v \in C^\infty(\R^d \setminus \{0\})\) that satisfies \eqref{eqFractionalSobolevAssumption} and for every \(x \in \R^d\), we have
\[
\bigl| v * \nabla\rho_r (x) \bigr| + \Bigl| v * \frac{\difq \rho_r}{\difq r} (x) \Bigr|
\le \frac{C''}{(r + |x|)^{d - s + 1}} \text{,}
\]
where 
\[
\rho_r(z)
\vcentcolon= \frac{1}{r^d} \rho\Bigl( \frac{z}{r}  \Bigr).
\]
\end{lemma}

\resetconstant
\begin{proof}[Proof of Lemma~\ref{lemmaEstimateConvolutionV}]
Since \(\rho\) is smooth and \(\rho_{r}\) is supported in \(B_{r}\), we have 
\[{}
\abs{\nabla\rho_r} + \Bigl| \frac{\difq \rho_r}{\difq r} \Bigr|
\le \frac{\C}{r^{d + 1}} \chi_{B_r},
\]
where \(\chi_{B_r}\) denotes the characteristic function of the ball \(B_{r}\).
For \(|x| \le 2r\), it then follows by the pointwise assumption of \(v\) that
\begin{equation}
	\label{eq-655}
\bigl| v * \nabla\rho_r (x) \bigr| + \Bigl| v * \frac{\difq \rho_r}{\difq r} (x) \Bigr|
\le \frac{\C}{r^{d + 1}} \int_{B_r(x)} |v(z)| \dif z
\le \frac{\C}{r^{d - s + 1}}.
\end{equation}
Since \(\rho_{r}\) is supported by \(B_{r}\), by the Divergence theorem we have
\[
\bigl| v * \nabla\rho_r (x) \bigr|
\le \int_{B_r} |\nabla v(x - y)| \rho_r(y) \dif y.
\]
We next observe that
\[
\frac{\difq \rho_r}{\difq r} (y)
= - \frac{1}{r} \Div{(y \rho_r(y))}.
\]
Another application of the Divergence theorem also gives
\[
\Bigl| v * \frac{\difq \rho_r}{\difq r} (x) \Bigr|
\le \int_{B_r} |\nabla v(x - y)| \rho_r(y) \dif y.
\]
Since \(\rho_{r} \le \Cl{cte-659}/r^{d}\), by the pointwise  assumption of \(\nabla v\) we get for \(|x| \ge 2r\), 
\begin{equation}
\label{eq-681}
\bigl| v * \nabla\rho_r (x) \bigr| + \Bigl| v * \frac{\difq \rho_r}{\difq r} (x) \Bigr|
\le \frac{\Cr{cte-659}}{r^{d}} \int_{B_r(x)} |\nabla v(z)| \dif z
\le \frac{\C}{|x|^{d - s + 1}}.
\end{equation}
The conclusion then follows from \eqref{eq-655} and \eqref{eq-681}.
\end{proof}

\resetconstant
\begin{proof}[Proof of Proposition~\ref{propositionFractionalSobolev}]
We first assume that \(F = (F_1, \ldots, F_d)\) is smooth, every \(\nabla F_{l}\) is summable in \(\R^{d}\)  and the linear condition \eqref{eqFractionalSobolevRotationalZero} holds pointwise in \(\R^d\).
We prove in this case that
\begin{equation*}
    \biggl| \int_{\R^d} v * F_1 \cdot \varphi \biggr|
    \le \C \|F\|_{L^1} \|\varphi\|_{L^{\frac{d}{s}}}
    \quad
    \text{for every \(\varphi \in C_c^\infty(\R^d)\).}
\end{equation*}
The estimate for the other components of \(F\) follows along the same lines.
Note that
\[
\int_{\R^d} v * F_1 \cdot \varphi
= \int_{\R^d} F_1 \cdot \widetilde v * \varphi,
\]
where \(\widetilde v(z) \vcentcolon= v(-z)\).{}
Given a mollifier \(\rho\) supported in \(B_1\) and \(\epsilon > 0\), 
for \(\varphi \in C_c^\infty(\R^d)\) and \(x_d \in \R\), following an idea of Van~Schaftingen~\cite{VS}*{Proof of Theorem~1.5} we estimate
\begin{multline*}
\biggl| \int_{\R^{d - 1}} F_1(x', x_d) \cdot \widetilde v * \varphi(x', x_d) \dif x' \biggr|\\
    \le
\biggl| \int_{\R^{d - 1}} F_1(x', x_d) \cdot \widetilde v * (\varphi - \rho_\epsilon * \varphi)(x', x_d) \dif x' \biggr|\\
+ \biggl| \int_{\R^{d - 1}} F_1 (x', x_d) \cdot \widetilde v * \rho_\epsilon * \varphi (x', x_d) \dif x' \biggr|.
\end{multline*}

We show that
\begin{multline}
\label{eq420}
\biggl| \int_{\R^{d - 1}} F_1(x', x_d) \cdot  \widetilde v * (\varphi - \rho_\epsilon * \varphi)(x', x_d) \dif x' \biggr|\\
\le \Cl{cte417}\epsilon^{\frac{s}{d}} \norm{F_1(\cdot, x_d)}_{L^1(\R^{d-1})}  \mathcal{M}\Phi(x_d)
\end{multline}
and
\begin{equation}
\label{eq428}
\biggl| \int_{\R^{d - 1}} F_1 (x', x_d) \cdot \widetilde v * \rho_\epsilon * \varphi (x', x_d) \dif x' \biggr|
\le \frac{\Cl{cte430}}{\epsilon^{\frac{d - s}{d}}} \norm{F_d}_{L^1(\R^d)} \mathcal{M}\Phi(x_d),
\end{equation}
where \(\Phi \colon  \R \to [0, \infty)\) is given by
\[
\Phi(x_d) \vcentcolon= \biggl( \int_{\R^{d - 1}} |\varphi(x', x_d)|^{\frac{d}{s}} \dif x'  \biggr)^{\frac{s}{d}}
\]
and \(\mathcal{M}\Phi  \colon  \R \to [0, \infty]\) is the maximal function associated to \(\Phi\),
\[
\mathcal{M}\Phi(t)
\vcentcolon= \sup_{r > 0}{\frac{1}{2r}\int_{t - r}^{t + r} \Phi}.
\]

We begin with \eqref{eq420}.
For every \(y \in \R^d\),
\[
\rho_\epsilon * \varphi(y) - \varphi(y)
 = \int_0^\epsilon  \frac{\difq\rho_r}{\difq r} * \varphi(y) \dif r.
\]
Thus, by Fubini's theorem,
\begin{equation}
    \label{eq-569}
\bigl| \widetilde v * (\rho_\epsilon * \varphi - \varphi)(x)\bigr| 
\le \int_0^\epsilon  \Bigl| \widetilde v * \frac{\difq\rho_r}{\difq r} * \varphi(x) \Bigr| \dif r.
\end{equation}
By Lemma~\ref{lemmaEstimateConvolutionV}, the pointwise assumptions on \(v\) and \(\nabla v\) yield, for every \(z \in \R^d\),
\[
\Bigl| \widetilde v * \frac{\difq \rho_r}{\difq r} (z) \Bigr| + \bigl| \widetilde v * \nabla\rho_r (z) \bigr|
\le \frac{\C}{(r + |z|)^{d - s + 1}},
\]
which by application of Hölder's inequality implies, for every \(r > 0\) and \(x = (x', x_d) \in \R^{d - 1} \times \R\),
\begin{equation}
\label{eq459}
\Bigl| \widetilde v * \frac{\difq \rho_r}{\difq r} * \varphi(x) \Bigr| + \bigl| \widetilde v * \nabla\rho_r * \varphi(x) \bigr|
\le \frac{\Cl{cte405}}{r^{\frac{d - s}{d}}} \mathcal{M}\Phi(x_d).
\end{equation}
A combination of \eqref{eq-569} and \eqref{eq459} gives
\[
\bigl| \widetilde v * (\rho_\epsilon * \varphi - \varphi)(x) \bigr|
\le \Cr{cte405} \int_0^\epsilon \frac{\dif r}{r^{\frac{d - s}{d}}} \mathcal{M}\Phi(x_d)
= \Cl{cte416}\epsilon^{\frac{s}{d}}  \mathcal{M}\Phi(x_d),
\]
which implies \eqref{eq420}.

We now show \eqref{eq428}.
Since \(d \ge 2\) and we are assuming that \(\partial_{d} F_{1} \in L^{1}(\R^d)\), for almost every \(x' \in \R^{d - 1}\) we have \(\partial_d F_1 (x', \cdot) \in L^1(\R)\) and then by the Fundamental theorem of Calculus, for every \(x_d \in \R\) we have
\[
F_1(x', x_d) 
= - \int_{x_d}^\infty \partial_d F_1 (x', t) \dif t.
\]
Since \(\partial_d F_1 = \partial_1 F_d\), we get by Fubini's theorem and integration by parts,
\begin{multline*}
\int_{\R^{d - 1}} F_1 (x', x_d) \cdot \widetilde v * \rho_\epsilon * \varphi (x', x_d) \dif x'\\
\begin{aligned}
& = - \int_{\R^{d - 1}}  \biggl(  \int_{x_d}^\infty \partial_1 F_d(x', t) \cdot \widetilde v * \rho_\epsilon * \varphi (x', x_d)  \dif t \biggr)\dif x' \\
& = \int_{x_d}^\infty \biggl( \int_{\R^{d - 1}} F_d(x', t) \cdot \widetilde v * \partial_1 \rho_\epsilon * \varphi (x', x_d) \dif x' \biggr) \dif t.
\end{aligned}
\end{multline*}
By \eqref{eq459},  we have
\[
\bigl| \widetilde v * \partial_1 \rho_\epsilon * \varphi (x', x_d) \bigr|
\le \frac{\Cr{cte405}}{\epsilon^{\frac{d - s}{d}}}  \mathcal{M}\Phi(x_d),
\]
which implies \eqref{eq428}.

Combining \eqref{eq420} and \eqref{eq428}, we get for every \(\epsilon > 0\),
\begin{multline*}
\biggl| \int_{\R^{d - 1}} F_1(x', x_d) \cdot  \widetilde v * \varphi (x', x_d) \dif x' \biggr|\\
\le \Bigl(\Cr{cte417} \epsilon^{\frac{s}{d}}  \norm{F_1(\cdot, x_d)}_{L^1(\R^{d - 1})} + \frac{\Cr{cte430}}{\epsilon^{\frac{d - s}{d}}}  \norm{F_d}_{L^1(\R^d)} \Bigr) \mathcal{M}\Phi(x_d).
\end{multline*}
Then, optimizing the right-hand side with respect to \(\epsilon\), we get for every \(x_d \in \R\),
\begin{multline*}
\biggl| \int_{\R^{d - 1}} F_1(x', x_d) \cdot \widetilde v * \varphi (x', x_d) \dif x' \biggr|\\
\le \Cl{cte-620} \norm{F_1(\cdot, x_d)}_{L^1(\R^{d - 1})}^{\frac{d - s}{d}} \norm{F_d}_{L^1(\R^d)}^{\frac{s}{d}} \mathcal{M}\Phi(x_d).
\end{multline*}
Integrating with respect to \(x_d\) and applying Hölder's inequality, we then get
\[
\biggl| \int_{\R^{d}} v * F_1 \cdot \varphi \biggr|
= \biggl| \int_{\R^{d}} F_1 \cdot \widetilde v * \varphi \biggr|
\le \Cr{cte-620} \norm{F_1}_{L^1(\R^d)}^{\frac{d - s}{d}} \norm{F_d}_{L^1(\R^d)}^{\frac{s}{d}} \norm{\mathcal{M}\Phi}_{L^{\frac{d}{s}}(\R)}.
\]
Since \(s < d\), we have the strong type estimate for \(\mathcal{M}\Phi\),
\[
\norm{\mathcal{M}\Phi}_{L^{\frac{d}{s}}(\R)}
\le \C \norm{\Phi}_{L^{\frac{d}{s}}(\R)}
= \norm{\varphi}_{L^{\frac{d}{s}}(\R^d)},
\]
the conclusion follows from the Riesz representation theorem by taking the supremum with respect to all \(\varphi \in C_c^\infty(\R^d)\) such that \(\norm{\varphi}_{L^{\frac{d}{s}}(\R^d)} \le 1\).

In the case where \(F\) merely belongs to \(L^{1}(\R^{d}; \R^{d})\), one may apply the inequality thus obtained with \(\rho_{r} * F\) for any \(r > 0\).{}
Note that in this case that \eqref{eqFractionalSobolevRotationalZero} holds pointwise by \(\rho_{r} * F\) in \(\R^d\) and, for every \(l \in \{1, \ldots, d\}\), since \(F \in L^{1}(\R^{d}; \R^{d})\), the gradient 
\[{}
\nabla(\rho_{r} * F_{l}) = (\nabla\rho_{r}) * F
\]
is summable in \(\R^{d}\).{}
We then have, for every \(r > 0\) and \(u \in C_{c}^{\infty}(\R^{d})\),
\[
\|v * (\rho_{r} * F)\|_{L^{\frac{d}{d - s}}}
\le C' \|\rho_{r} * F\|_{L^1}
\le C' \|F\|_{L^1}.
\]
As \(r \to 0\), the conclusion then follows from Fatou's lemma.
\end{proof}

The constant \(C' > 0\) in the statement of Proposition~\ref{propositionFractionalSobolev} is under control as long as \(s/d\) stays away from \(0\) and \(1\).
In particular, it can be taken independently of \(s\) when \(1/2 \le s \le 1\).
As a result,

\begin{proof}[Proof of Theorem~\ref{propositionEstimateRieszFractionalGradient}]
    In dimension \(d \ge 2\), the representation formula \eqref{eq:fundamental_theorem} holds for every \(u \in C_c^\infty(\R^d)\).
    We rely on the explicit formula of the constant \(A = A(s, d)\) such that \(V_s = A x/\abs{x}^{d - s + 1}\) that can be obtained using the Fourier transform of homogeneous functions.
    In fact, one finds that
    \[
    A = K(s, d) (1 - s)
    \quad \text{where} \quad
    K(s, d) \vcentcolon= \frac{1}{\sigma_d \pi^{d/2}} \frac{\Gamma(\frac{d+1+s}{2}) \Gamma(\frac{d+1-s}{2})}{\Gamma(\frac{d}{2}) \Gamma(\frac{3 - s}{2}) \Gamma(\frac{s + 1}{2})}.   
    \]
    We next apply Proposition~\ref{propositionFractionalSobolev} with \(F = \nabla^s u\) and \(v\) given by the components of \(V_s/A\).
    Estimate \eqref{eqFractionalSobolevAssumption} is then verified with a constant \(C > 0\) that is independent of \(1/2 \le s < 1\) and we find 
    \[
    \frac{1}{A} \|u\|_{L^{\frac{d}{d-s}}} 
    \le \sum_{j = 1}^d{\|V_{s, j} * \nabla^s u\|_{L^{\frac{d}{d-s}}}}
    \leq d C' \|\nabla^s u\|_{L^1}.
    \]
    Hence, by the form of \(A\),
    \[
    \|u\|_{L^{\frac{d}{d-s}}} 
    \leq d C' K(s, d) (1 - s) \|\nabla^s u\|_{L^1},
    \]
    which gives the conclusion since \(K(s, d)\) can be bounded by a constant independent of \(1/2 \le s < 1\).
\end{proof}

\section{Characterization of radial homogeneous distributions}
\label{sectioncharacterizationradialhomogeneous}
We say that a distribution \(T\) is represented in an open set \(\Omega\) by a function \(f \in L^1\loc(\Omega; \CC)\) whenever
\begin{equation}
\label{eq-1329}
\langle T, \varphi \rangle 
= \int_{\R^d} f \varphi
\quad \text{for every \(\varphi \in C_c^\infty(\Omega)\).}
\end{equation}
We then use the notation \(\DistribF{T} \vcentcolon= f\) in \(\Omega\). The aim of this section is to identify radial homogeneous distributions that are represented in $\R^N$ by a locally integrable function.

\begin{definition}
Let \(T\) be a distribution in \(\R^{d} \setminus \{0\}\).
We say that \(T\) is \emph{radial} whenever, for each orthogonal transformation \(R \in O(d)\),
\[{}
\langle T, \varphi \circ R \rangle{}
= \langle T, \varphi \rangle{}
\quad \text{for every \(\varphi \in C_{c}^{\infty}(\R^{d} \setminus \{0\})\).} 
\]
Given \(\lambda \in \R\), we say that \(T\) is \emph{homogeneous of order \(\lambda\)} whenever, for each \(t > 0\),{}
\[{}
\langle T, \varphi_{t} \rangle{}
= t^{\lambda} \langle T, \varphi \rangle
\quad \text{for every \(\varphi \in C_{c}^{\infty}(\R^{d}  \setminus \{0\})\),} 
\]
where \(\varphi_{t}\) is the function defined by
\begin{equation}
	\label{eq-scaling}
	\varphi_{t}(x) = \frac{1}{t^{d}} \varphi\Bigl( \frac{x}{t} \Bigr).
\end{equation}
\end{definition}
Note, for example, that the distribution associated to \(\abs{x}^{\lambda}\) in \(\R^{d} \setminus \{0\}\) is radial and homogeneous of order \(\lambda \in \R\). 
These definitions have immediate counterparts for distributions in \(\R^{d}\) and \(\abs{x}^{\lambda}\) defines a distribution in \(\R^{d}\) for \(\lambda > -d\).

We present a proof of the following proposition based on an unpublished note by E.~Y.~Jaffe~\cite{Jaffe}:

\begin{proposition}
	\label{propositionCharacterization}
	Let \(d \ge 2\).{}
	If \(T\) is a radial homogeneous distribution of order \(0\) in \(\R^{d} \setminus \{0\}\), then \(T\) is constant.
    More precisely, there exists a constant \(C \in \R\) such that
    \[
    \langle T, \varphi \rangle
    = C \int_{\R^d \setminus \{0\}} \varphi
    \quad \text{for every \(\varphi \in C_c^\infty(\R^d \setminus \{0\})\).}
    \]
\end{proposition}

We begin with a counterpart of Euler's homogeneous condition for distributions of order zero:

\begin{lemma}
	\label{lemmaHomogeneous}
	If \(T\) is a distribution of order \(0\) in \(\R^{d} \setminus \{0\}\), then
	\[{}
	\sum_{j = 1}^{d}{x_{j} \partial_{j}T} = 0
	\quad \text{in \(\R^{d} \setminus \{0\}\).}
	\]
\end{lemma}

\begin{proof}[Proof of Lemma~\ref{lemmaHomogeneous}]
	Let \(\varphi \in C_{c}^{\infty}(\R^{d} \setminus \{0\})\) and \(\varphi_{t}\) be given by \eqref{eq-scaling}.
	Since \(T\) is homogeneous of order \(0\), for every \(t > 0\) we have
	\begin{equation}
	\label{eq-903}
	\langle T, \varphi_{t} \rangle{}
	= \langle T, \varphi \rangle{}.
	\end{equation}
	Observe that
	\[{}
	\frac{\difq }{\difq t} \varphi_{t}(x) \Big|_{t = 1}
	= - d\varphi(x) - \nabla \varphi(x) \cdot x 
	= - \Div{(\varphi(x) x)}.
	\]
	Thus, differentiating both sides of \eqref{eq-903} with respect to \(t\) at \(1\), one gets
	\[{}
	  \Bigl\langle \sum_{j = 1}^{d} x_{j} \partial_{j} T, \varphi \Bigr\rangle{}
	= - \langle T_{x}, \Div{(\varphi(x) x)} \rangle{}
	= 0.
	\qedhere
	\]
\end{proof}

We now prove that a radial ditribution has zero derivative with respect to vector fields that are tangential to spheres centered at \(0\): 

\begin{lemma}
	\label{lemmaRadial}
	If \(T\) is a radial distribution in \(\R^{d} \setminus \{0\}\), then for every smooth functions \(c_{1}, \ldots, c_{d}\) such that \(\sum\limits_{j=1}^{d}{c_{j}(x) x_{j}} = 0\) in \(\R^{d} \setminus \{0\}\), we have
	\[{}
	\sum_{j = 1}^{d}{c_{j} \partial_{j}T} = 0
	\quad \text{in \(\R^{d} \setminus \{0\}\).}
	\]
\end{lemma}

\begin{proof}[Proof of Lemma~\ref{lemmaRadial}]
	Let \((\rho_{\epsilon})_{\epsilon > 0}\) be a family of radial mollifiers in \(C_{c}^{\infty}(\R^{d})\) such that \(\rho_{\epsilon}\) is supported in the ball \(B_{\epsilon}\). 
	Given \(r > 0\), for every \(x \in \R^{d} \setminus \overline B_{r}\) and \(0 < \epsilon < r\) the function \(y \mapsto \rho_{\epsilon}(x - y)\) is supported in \(\R^{d} \setminus \{0\}\) and then the convolution
	\[{}
	\rho_{\epsilon} * T(x)
	\vcentcolon= \langle T_{y}, \rho_{\epsilon}(x - y) \rangle
	\] 
	is well-defined and smooth in \(\R^{d} \setminus \overline B_{r}\).{}
	Given \(R \in O(d)\), by radiality of \(T\) and \(\rho\) we have
	\[{}
	\rho_{\epsilon} * T(Rx)
	= \langle T_{y}, \rho_{\epsilon}(Rx - y) \rangle
	= \langle T_{y}, \rho_{\epsilon}(Rx - Ry) \rangle{}
	= \langle T_{y}, \rho_{\epsilon}(x - y) \rangle{}
	= \rho_{\epsilon} * T(x).
	\]
	Thus, \(\rho_{\epsilon} * T\) is a radial function in \(\R^{d} \setminus \overline B_{r}\).{}

	By the assumption on the functions \(c_{1}, \ldots, c_{d}\), the vector \(v \vcentcolon= (c_{1}(x), \ldots, c_{d}(x))\) belongs to the tangent plane of the sphere \(\partial B_{\abs{x}}\) at \(x\).{}
	Thus, there exists a smooth curve \(\gamma \colon (-1, 1) \to \partial B_{\abs{x}}\)  such that \(\gamma(0) = x\) and \(\gamma'(0) = v\).{}
	In particular, \((\rho_{\epsilon} * T) \circ \gamma\) is constant and applying the chain rule we get
	\[{}
	\sum_{j = 1}^{d} c_{j}(x) \partial_{j}(\rho_{\epsilon} * T)(x)
	= \frac{\difq }{\difq t} (\rho_{\epsilon} * T) \circ \gamma(t) \Big|_{t = 0}
	= 0.
	\]
	As \(\epsilon \to 0\), the family of functions \(\sum\limits_{j = 1}^{d} c_{j} \partial_{j}(\rho_{\epsilon} * T)\) converges weakly in the sense of distributions in \(\R^{d} \setminus \overline B_{r}\) to \(\sum\limits_{j = 1}^{d}{c_{j} \partial_{j}T}\).{}
	Hence, 
	\[{}
	\sum_{j = 1}^{d}{c_{j} \partial_{j}T} = 0
	\quad \text{in \(\R^{d} \setminus \overline B_{r}\)\,.}
	\]
	Since this property holds for any \(r > 0\), the conclusion follows.
\end{proof}

\begin{proof}[Proof of Proposition~\ref{propositionCharacterization}]
	Let \(k \in \{1, \ldots, d\}\).{}
	There exist smooth functions \(p \colon \R^{d} \setminus \{0\} \to \R\) and \(w  \colon \R^{d} \setminus \{0\} \to \R^{d}\) such that, for every \(x \ne 0\),
	\begin{equation*}
	e_{k} = p(x) \sum_{j = 1}^{d}{x_{j}e_{j}} + w(x)
	\quad \text{and} \quad 
	w(x) \cdot x = 0.
	\end{equation*}
	Indeed, it suffices to take \(p(x) \vcentcolon= x \cdot e_{k}/\abs{x}^{2}\). 
	Writing \(w = \sum\limits_{j = 1}^{d} c_{j} e_{j}\) for smooth functions \(c_{1}, \ldots, c_{d}\) in \(\R^{d} \setminus \{0\}\), then by Lemmas~\ref{lemmaHomogeneous} and~\ref{lemmaRadial} we have
	\[{}
	\sum_{j = 1}^{d}{x_{j} \partial_{j}T} = 0
	\quad \text{and} \quad 
	\sum_{j = 1}^{d}{c_{j} \partial_{j}T} = 0
    \quad \text{in \(\R^{d} \setminus \{0\}\).}
	\]
	Thus,  for every \(k \in \{1, \ldots, d\}\),
	\[{}
	\partial_{k}T
	= p \sum_{j = 1}^{d}{x_{j} \partial_{j}T} + \sum_{j = 1}^{d}{c_{j} \partial_{j}T} = 0
    \quad \text{in \(\R^{d} \setminus \{0\}\).}
	\]
	Since \(d \ge 2\), the set \(\R^{d} \setminus \{0\}\) is connected and we conclude that \(T\) is constant.
\end{proof}

Observe that when \(f\) is a homogeneous function of order \(\gamma > -d\), then \(f \in L^{1}\loc(\R^{d})\) and one may simply take as \(T\) the distribution associated to \(f\) by integration, more precisely, for every \(\varphi \in C_{c}^{\infty}(\R^{d})\),
\[
\langle T, \varphi \rangle{}
= \int_{\R^{d}} f \varphi.
\]
Note in this case that \(\DistribF{T} = f\) in \(\R^{d}\).

\begin{corollary}
	\label{corollaryDistributionCharacterization}
    Let \(d \ge 2\) and  \(\lambda > -d\).{}
	If\/ \(T\) is a radial homogeneous distribution of order \(\lambda\) in \(\R^{d}\), then \(T\) is represented by the function \(C\abs{x}^{\lambda}\) in \(\R^{d}\) for some constant \(C \in \R\), that is, \[{}
	\DistribF{T} = C\abs{x}^{\lambda}
	\quad \text{in \(\R^{d}\).}
	\]
\end{corollary}

\begin{proof}
	Since the function \(\abs{x}^{-\lambda}\) is smooth in \(\R^{d} \setminus \{0\}\), \(\abs{x}^{-\lambda} T\) is well-defined as a distribution in \(\R^{d} \setminus \{0\}\).{}
	This distribution is radial and homogeneous of order \(0\), whence by Proposition~\ref{propositionCharacterization}, it is represented by a constant in \(\R^{d} \setminus \{0\}\).{}
	We deduce that \(T\) is represented by  
	\(C \abs{x}^{\lambda}\) in \(\R^{d} \setminus \{0\}\).{}
	Since \(\lambda > -d\), we have \(C \abs{x}^{\lambda} \in L^{1}\loc(\R^{d})\) and then this function represents a distribution \(S\) in \(\R^{d}\).{}
	It thus follows that the distribution \(T - S\) is supported by \(\{0\}\), whence it is a finite combination of a Dirac mass and its derivatives.
	Since these are all homogeneous distributions of order less than or equal to \(-d\) and \(T - S\) is homogeneous of order \(\lambda\) greater than \(-d\), we deduce that \(T - S = 0\) in \(\R^{d}\).
\end{proof}

\section{Fourier transform of radial homogeneous distributions}
\label{sectionfouriertransformradialhomogeneous}
The Fourier transform of \(f \in L^1(\R^d)\) is the bounded continuous function \(\widehat{f} \colon  \R^d \to \CC\) defined by
\[
\widehat{f}(\xi)
= \int_{\R^d} \e^{-2\pi \imath x \cdot \xi} f(x) \dif x
\]
and the inverse Fourier transform \(\widecheck{f} \colon  \R^d \to \CC\) is
\[
\widecheck{f}(x)
= \int_{\R^d} \e^{2\pi \imath x \cdot \xi} f(\xi) \dif \xi.
\]
More generally, for a tempered distribution \(T\) in \(\R^d\), we denote by \(\Fourier{T}\) the distributional Fourier transform defined for every Schwartz function \(\eta\) in \(\R^{d}\) by
\[
\langle \Fourier{T}, \eta \rangle 
= \langle T, \widehat{\eta} \rangle{}
\]
and, by analogy, one defines the inverse Fourier transform \(\Fourier^{-1}{T}\) using \(\widecheck{\varphi}\).
Observe that \(\widehat{\varphi}\) and \(\widecheck{\varphi}\) are Schwartz functions in \(\R^d\), whence \(\Fourier{T}\) and \(\Fourier^{-1}{T}\) are well-defined.

\begin{example}
\label{exampleFourierRepresentation}
If a tempered distribution \(T\) is represented by a function in \(\R^d\) and \(\DistribF{T} \in L^1(\R^d)\), then it follows from Fubini's theorem that \(\Fourier{T}\) is represented by the Fourier transform of \(\DistribF{T}\),
\[
\DistribF{\Fourier{T}} = \widehat{\DistribF{T}}
\quad \text{in \(\R^d\).}
\]
For \(\DistribF{T} \in L^2(\R^d)\), not necessarily summable, \(\DistribF{\Fourier{T}}\) is the \(L^2\)~Fourier-Plancherel transform of \(T\).
More generally, assume that \(\DistribF{T}\) is merely a locally integrable function in \(\R^d\) with at most polynomial growth at infinity.
In this case, one can compute \( \DistribF{\Fourier{T}} \) by approximation of \(\DistribF{T}\) as follows,
\begin{equation}
\label{eq-1384}
\langle \Fourier{T}, \eta \rangle 
= \lim_{r \to \infty}{\int_{\R^d}} \widehat{\DistribF{T} \chi_{B_r}} \,\eta,
\end{equation}
where \(\eta\) is any  Schwarz function.
Indeed, since \(T\) is a tempered distribution and \(\DistribF{T}\) has at most polynomial growth at infinity, we may approximate \(\eta\) by smooth functions with compact support and deduce the companion identity to \eqref{eq-1329} for tempered distributions,
\[
\langle T, \eta \rangle = 
\int_{\R^d} \DistribF{T} \, \eta.
\]
As a result, replacing in this identity \(\eta\) by \(\widehat{\eta}\), by the Dominated convergence theorem we get
\[
\langle \Fourier{T}, \eta \rangle
= \langle T, \widehat{\eta} \rangle 
= \int_{\R^d} \DistribF{T} \, \widehat{\eta}
= \lim_{r \to \infty}{\int_{\R^d} \DistribF{T}\chi_{B_{r}} \, \widehat{\eta}},
\]
which immediately implies \eqref{eq-1384} by a standard property of the Fourier transform of a function in \(L^1(\R^d)\). 
\end{example}

Observe that if \(T\) is a tempered radial distribution in \(\R^d\) that is homogeneous of degree \( \alpha - d \), then \(\Fourier{T}\) is a tempered radial distribution in \(\R^d\) that is homogeneous of degree \( - \alpha \).
We may thus rely on a large class of homogeneous functions in \(\R^d\) to construct distributions whose Fourier transforms can be easily identified by Corollary~\ref{corollaryDistributionCharacterization}.

\begin{example}
\label{exampleHomogeneous}
Given \(0 < \alpha < d\), consider the homogeneous function \(f_\alpha\) of order \(\alpha - d\) defined for \(x \ne 0\) by 
\begin{equation}
\label{eqHomogenousFunction}
f_\alpha(x) = \frac{1}{|x|^{d - \alpha}},
\end{equation}
which belongs to \(L^1\loc(\R^d)\) and defines a tempered distribution \(T_\alpha\) in \(\R^d\).
Note that \(T_\alpha\) is radial and homogeneous of order \(\alpha - d\).
Hence, \(\Fourier{T_\alpha}\) is radial and homogeneous of order \(-\alpha\) in \(\R^d\).
Therefore, by Corollary~\ref{corollaryDistributionCharacterization}, we deduce that
\begin{equation}
\label{eqHomogeneousFourier}
\DistribF{\Fourier{T_\alpha}}(\xi)
= \frac{c_{\alpha}}{|\xi|^{\alpha}}
\quad \text{in \(\R^{d}\)}.
\end{equation}
One then shows that
\(
c_{\alpha}=\pi^{\frac d2-\alpha}{\Gamma(\frac{\alpha}{2})}/{\Gamma (\frac{d-\alpha}{2})}
\), see \cite{SKM}*{p.\,490}.
\end{example}

If $\alpha\le0$ in \eqref{eqHomogenousFunction}, then $f_{\alpha}$ does not belong to \(L^1\loc(\R^d)\). However, let us recall a remarkable property of homogeneous functions in \(\R^{d} \setminus \{0\}\) proved by Hörmander~\cite{Hormander}*{Theorems~3.2.4 or~3.2.3}:

\begin{theorem}\label{HormanderNew}
For every homogeneous function $f \in C^\infty(\mathbb{R}^d\setminus\{0\})$, there exists a tempered distribution $T$ in $\mathbb{R}^d$  that is represented by \(f\) in \(\R^d \setminus \{0\}\), in other words,
\[
\DistribF{T} = f
\quad \text{in \(\R^d \setminus \{0\}\).}
\]
\end{theorem}  
The proof of Theorem~\ref{HormanderNew} is based on the explicit construction of the tempered distribution and one can follow the argument to construct suitable $T_{\alpha}$ with $\alpha\le0$, see the Examples below. Concerning the Fourier transform of a tempered distribution, we also recall \cite{Hormander}*{Theorem~7.1.18}:

\begin{theorem}\label{HormanderFourier}
If a distribution \(T\) in \(\R^d\) is represented in \(\R^d \setminus \{0\}\) by a smooth homogeneous function \(f \colon \mathbb{R}^d\setminus\{0\} \to \CC\), then \(T\) is a tempered distribution and \(\Fourier{T}\) is also represented in \(\R^d \setminus \{0\}\) by a smooth homogeneous function \(h  \colon \mathbb{R}^d\setminus\{0\} \to \CC\).
\end{theorem}

Note that Theorem~\ref{HormanderFourier} does not provide a clear relation between the functions that represent \(T\) and \(\Fourier{T}\). Here we rely on the characterization of the previous section to compute the Fourier transform of \(T_{\alpha}\) for \(\alpha < 0\) that is not an even integer.

\begin{example}
\label{exampleFourierHomogeneousNegative}
When \(\alpha \le 0\) is not an integer, let \(T_\alpha\) be  the tempered distribution defined for every Schwarz function \(\eta\) by 
\begin{equation}
\label{eq-952}
\langle T_{\alpha}, \eta \rangle = \int_{\R^{d}} \frac{1}{\abs{x}^{d - \alpha}} (\eta(x) - P_{0}^{k}\eta (x)) \dif x,
\end{equation}
where \(P_{0}^{k}\eta\) is the Taylor polynomial of \(\eta\) of order \(k \in \N\) at \(0\).
We take \(k\) such that \(k + 1 + \alpha > 0\) to ensure summability in a neighborhood of the origin and \(k + \alpha < 0\) for summability near infinity. For \eqref{eq-952} we refer to the already mentioned proof of \cite{Hormander}*{Theorems~3.2.4} or  \cite{SKM}*{Eq.~(25.22)}. 
Note that \(T_\alpha\) is represented in \(\R^d \setminus \{0\}\) by \(f_\alpha\) given by \eqref{eqHomogenousFunction}. 
This distribution \(T_\alpha\) is radial and homogeneous of order $\alpha-d$.
Its Fourier transform \(\Fourier{T_\alpha}\) is a radial homogeneous distribution of order \(-\alpha\).
To see why its homogeneous, one first observes that 
\[
\widehat{\eta_t}(x) = \widehat{\eta}\Bigl( \frac{x}{t} \Bigr)
\quad \text{and} \quad
P_{0}^{k} \widehat{\eta_t}(x)
= P_{0}^{k}\widehat{\eta}\Bigl( \frac{x}{t} \Bigr).
\]
By Corollary~\ref{corollaryDistributionCharacterization}, \(\Fourier{T_\alpha}\) is therefore represented in \(\R^d\) by the \(L^1\loc\)~function \eqref{eqHomogeneousFourier}.
\end{example}

\begin{example}\label{exampleFourier-2k}
Assume that \(\alpha \le 0\) is an integer and take \(k \vcentcolon= -\alpha\). 
When \(k = -\alpha\) is odd, we may take \eqref{eq-952} as the definition of \(T_\alpha\) by considering the integral as a principal value, more precisely,
\[
\langle T_{\alpha}, \eta \rangle 
= \lim_{r \to \infty}{\int_{B_r} \frac{1}{\abs{x}^{d - \alpha}} (\eta(x) - P_{0}^{k}\eta (x)) \dif x}.
\]
Summability near infinity is then ensured by the fact that the term of order \(k\) in the Taylor polynomial is an odd function and therefore its principal value integral is equal to zero.
Since \(\Fourier{T_\alpha}\) is radial and homogeneous of order \(-\alpha\), it is also represented in \(\R^d\) by the \(L^1\loc\)~function \eqref{eqHomogeneousFourier}.

When \(k = -\alpha\) is even, it is not possible to define a homogeneous tempered distribution in \(\R^d\) that can be represented by \(1/\abs{x}^{d - \alpha}\) in \(\R^d \setminus \{0\}\) and one has to waive the homogeneity property.
An alternative in this case is take \(T_\alpha\) defined for every Schwartz function \(\eta\) by
\[
\langle T_{\alpha}, \eta \rangle 
=\frac{-1}{k!}\biggl[\int_0^{\infty}\log{\rho} \; \Average{\eta}{}^{(k+1)}(\rho) \dif\rho+\Average{\eta}{}^{(k)}(0) \sum_{j=1}^{k}\frac{1}{j}\biggr],
\]
where
\begin{equation*}
    \Average{\eta}{}(\rho) 
    \vcentcolon= \int_{\partial B_1} \eta(\rho y) \dif \sigma(y).
\end{equation*}
The distribution \(T_\alpha\) is not homogeneous but is represented in \(\R^d \setminus \{0\}\) by \(f_\alpha\) given by \eqref{eqHomogenousFunction}.
In fact, if $\varphi\in C^{\infty}_c(\mathbb{R}^d\setminus\{0\})$ then, integrating by parts, one gets
\[
\langle T_\alpha, \varphi\rangle
= -\frac{1}{k!} \int_0^{\infty}\log{\rho} \; \Average{\varphi}{}^{(k+1)}(\rho) \dif\rho 
=\int_{\mathbb{R}^d} \frac{\varphi(x)}{|x|^{d + k}} \dif x.
\]
The Fourier transform of such a distribution can be explicitly evaluated for every Schwarz function \(\eta\) as
\begin{equation}
\label{eqFourierCritical}
\langle\Fourier{T_\alpha}, \eta\rangle
= \int_{\mathbb{R}^d}(-A_k\log{|\xi|} + \lambda_{k})|\xi|^{k} \eta(\xi) \dif\xi,
\end{equation}
where
\begin{equation}
\label{eq921}
A_k \vcentcolon=(-1)^{\frac k2}\frac{2\pi^{k+\frac d2}}{\Gamma (\frac{N+k}{2})(k/2)!} \quad \mbox{and} \quad 
\lambda_k 
\vcentcolon= \frac{\difq}{\difq \alpha}\bigl((\alpha+k)c_{\alpha}\bigr)\Big|_{\alpha = -k}\,.
\end{equation}
We refer to \cite{SKM}*{Lemma~25.2 and Remark~25.2} for a proof of \eqref{eqFourierCritical}.
\end{example}

\section{Positivity of the Fourier transform}
\label{sectionPositive}

The follow result concerning the positivity of the Fourier transform is based on a standard argument, see e.g.~\cite{BCMC}*{Section~5.1}. 
We present a proof for the convenience of the reader:

\begin{proposition}
    \label{propositionPositivity}
    Let \(T\) be a tempered distribution represented in \(\R^d\) by a radial function \(\DistribF{T} \in C^\infty(\R^d \setminus \{0\})\) that is locally integrable in \(\R^d\) and is such that, for any \(x \ne 0\), the function
    \[
    \rho \in (0, \infty) \longmapsto \rho^{d - 1}\DistribF{T}(\rho x)
    \]
    is non-increasing, convex, and converges to zero as \(\rho \to \infty\).
    If \(\Fourier{T}\) is represented by a function \(\FourierF{T}\) in \(\R^d\) which is smooth in \(\R^d \setminus \{0\}\) and \(T\) is not identically zero, then \(\DistribF{\Fourier{T}} > 0\) in \(\R^d \setminus \{0\}\).
\end{proposition}

We rely on the following elementary property:

\begin{lemma}
    \label{lemmaPositiveElementary}
    Let \(\Lambda\) be a locally integrable function in \([0, \infty)\).
    If \(\Lambda\) is nonnegative, non-increasing and converges to zero at infinity, then, for every \(r > 0\),
    \begin{equation}
    \label{eq-1596}
    \int_0^r \sin{(2\pi\tau)} \Lambda(\tau) \dif \tau \ge 0
    \end{equation}
    and the the integral converges in the extended interval \([0, \infty]\) as \(r \to \infty\).
\end{lemma}

\begin{proof}[Proof of Lemma~\ref{lemmaPositiveElementary}]
    Since \(\Lambda\) is nonnegative, for \(r \in (k, k+1/2]\) with \(k \in \N\) we have
    \[
    \int_0^r \sin{(2\pi\tau)} \Lambda(\tau) \dif \tau
    \ge \int_0^k \sin{(2\pi\tau)} \Lambda(\tau) \dif \tau,
    \]
    whereas for \(r \in (k+1/2, k+1]\) with \(k \in \N\) we have
    \[
    \int_0^r \sin{(2\pi\tau)} \Lambda(\tau) \dif \tau
    \ge \int_0^{k+1} \sin{(2\pi\tau)} \Lambda(\tau) \dif \tau.
    \]
    By additivity of the integral, it thus suffices to show that, for every \(j \in \N\),
    \[
    \int_j^{j + 1} \sin{(2\pi\tau)} \Lambda(\tau) \dif \tau
    \ge 0.
    \]
    Since \(\Lambda\) is nonnegative and non-increasing, 
    \[
    \int_j^{j + 1} \sin{(2\pi\tau)} \Lambda(\tau) \dif \tau
    \ge \Lambda(j + 1/2) \int_j^{j + 1} \sin{(2\pi\tau)} \dif \tau = 0.
    \]
    Thus, the sequence of integrals in \eqref{eq-1596} with \(r = j \in \N\) is nondecreasing, whence its limit
    \[
    \lim_{j \to \infty}{\int_0^j \sin{(2\pi\tau)} \Lambda(\tau) \dif \tau}
    \]
    exists in \([0, \infty]\).
    To conclude the existence of the limit of the integrals in \eqref{eq-1596} as \(r \to \infty\), it suffices to observe that for every \(r \in [j, j+1]\) we have
    \[
    \biggl| \int_j^r \sin{(2\pi\tau)} \Lambda(\tau) \dif \tau  \biggr|
    \le (r - j)\Lambda(j) \le \Lambda(j)
    \]
    and that by assumption the right-hand side converges to zero as \(j \to \infty\).
\end{proof}

\begin{proof}[Proof of Proposition~\ref{propositionPositivity}]
By \eqref{eq-1384}, we have
\begin{equation}
\label{eq-1615}
\langle \Fourier{T}, \eta \rangle 
= \lim_{r \to \infty}{\int_{\R^d} \widehat{f_r} \,\eta},
\end{equation}
where \(f \vcentcolon= \DistribF{T} \) and \(f_r \vcentcolon= f \, \chi_{B_r}\).
Since \(f_r\) is radial, and in particular even, applying the integration formula in polar coordinates we get
\[
\widehat{f_r}(\xi)
= \int_{B_r} \cos{(2\pi x \cdot \xi)  f(x) \dif x}
= \int_{\partial B_1} \biggl( \int_{0}^r \cos{(2\pi \rho y \cdot \xi)  \vartheta(\rho) \rho^{d - 1} \dif \rho \biggr) \dif \sigma(y)},
\]
where \(\vartheta \in C^{\infty}(0, \infty)\) is such that \(\vartheta(\abs{x}) = f(x)\) for every \(x \ne 0\).

We show that the innermost integral is positive for every \(y \in \partial B_1\) and \(\xi \in \R^d\). 
It is enough to consider the case where \(\alpha \vcentcolon= y \cdot \xi \ne 0\).
Denote \(\Theta(\rho) = -\rho^{d - 1}\vartheta(\rho)\), so that
\[
\int_{0}^r \cos{(2\pi \rho y \cdot \xi)}  \vartheta(\rho) \rho^{d - 1} \dif \rho
= - \int_{0}^r \cos{(2\pi \rho \alpha)}  \Theta(\rho) \dif \rho.
\]
Integrating by parts, for every \(0 < \epsilon < r\)  we get
\begin{multline*}
\int_{\epsilon}^r \cos{(2\pi \rho y \cdot \xi)}  \vartheta(\rho) \rho^{d - 1} \dif \rho\\
= - \frac{1}{2 \pi |\alpha|} \sin{(2\pi \rho |\alpha|)} \Theta(\rho) \Big|_{\rho = \epsilon}^r + 
\frac{1}{2 \pi |\alpha|} \int_\epsilon^r \sin{(2\pi \rho |\alpha|)} \Theta'(\rho) \dif \rho.
\end{multline*}
Note that by assumption the function \(\Theta'\) is nonnegative, whence by the Fundamental theorem of calculus it is locally integrable in \([0, \infty)\).
On the other hand, since \(f\) is locally integrable, we have \(\int_0^1 \Theta(\rho)\rho \dif \rho/\rho < \infty\). Thus, there exists a sequence \((\epsilon_n)_{n \in \N}\) of positive numbers in \((0, 1)\) that converges to zero and satisfies \(\Theta(\epsilon_n)\epsilon_n \to 0\).
Taking \(\epsilon = \epsilon_n\) and letting \(n \to \infty\), we then get
\begin{multline*}
\int_{0}^r \cos{(2\pi \rho y \cdot \xi)}  \vartheta(\rho) \rho^{d - 1} \dif \rho\\
= - \frac{1}{2 \pi |\alpha|} \sin{(2\pi r |\alpha|)} \Theta(r) + 
\frac{1}{2 \pi |\alpha|} \int_0^r \sin{(2\pi \rho |\alpha|)} \Theta'(\rho) \dif \rho.
\end{multline*}

By assumption, \(\Theta\) converges to zero at infinity.
Hence, the first term in the right-hand side converges to zero as \(r \to \infty\), uniformly with respect to \(\alpha\).
Moreover, by Lemma~\ref{lemmaPositiveElementary}, the second term is nonnegative and converges as \(r \to \infty\).
Hence, by the Dominated convergence theorem and Fatou's lemma, for every nonnegative Schwarz function \(\eta\) we get
\[
\lim_{r \to \infty}{\int_{\R^d} \widehat{f_r} \,\eta}
\ge \int_{\R^d}  \biggl(\int_{\partial B_1} \frac{1}{2 \pi |y \cdot \xi|} \int_0^\infty \sin{(2\pi \rho |y \cdot \xi|)} \Theta'(\rho) \dif \rho \dif y \biggr)\eta(\xi) \dif \xi.
\]
Since \eqref{eq-1615} holds and \(\Fourier{T}\) is represented by a smooth function, we deduce that
\[
\DistribF{\Fourier{T}}(\xi)
\ge \int_{\partial B_1} \frac{1}{2 \pi |y \cdot \xi|} \int_0^\infty \sin{(2\pi \rho |y \cdot \xi|)} \Theta'(\rho) \dif \rho \dif y \ge 0.
\]
Since the function \(\Theta'\) is non-trivial, an inspection of the proof of Lemma~\ref{lemmaPositiveElementary} shows that the integral with respect to \(s\) is positive and the conclusion follows.
\end{proof}

\section{Fourier transform of radial doubly homogeneous functions}

\label{sectionfouriertransformdoublyhomogeneous}
In this section, we consider tempered distributions that are represented by functions that are homogeneous in a neighborhood of the origin and at infinity, but the homogeneity rates need not be the same. 

\begin{proposition}
\label{propositionFourierNonHomogeneous}
Take \(0 < \alpha < d\) and \(\beta < d \) with $\beta \not\in -2\N$.
Let \(g \in C^{\infty}(\R^d \setminus \{0\})\) be a radial function such that
\[
g(x) = 
\begin{cases}
{a}/{|x|^{d - \alpha}}
& \text{if\/ \(|x| \le r\),}\\
{b}/{|x|^{d - \beta}}
& \text{if\/ \(|x| \ge R\),}
\end{cases}
\]
where \(0 < r < R\) and \(a, b > 0\).
Then, there exists a tempered distribution \(\DistribS{g}\) in \(\R^{d}\) that is represented in \(\R^{d}\) by \(g\) and its Fourier transform \(\Fourier{\DistribS{g}}\) is represented in \(\R^{d}\) by a radial function \(\FourierF{\DistribS{g}} \in C^{\infty}(\R^d \setminus \{0\})\) that satisfies
\[
\FourierF{\DistribS{g}}(\xi)
= 
\begin{cases}
bc_{\beta}/|\xi|^{\beta}
+\eta(\xi) 
& \text{for \(|\xi|\le r\),}\\
ac_{\alpha}/|\xi|^{\alpha} + \zeta(\xi) 
& \text{for \(|\xi|\ge R\),}
 \end{cases}
\]
where \(\eta\) and \(\zeta\) are Schwartz functions.
\end{proposition}

We rely on the following

\begin{lemma}\label{fouriertransform}
If a tempered distribution \(T\) in \(\R^d\) is represented in \(\R^d \setminus \{0\}\) by a homogeneous function \(f \in C^\infty(\R^d \setminus \{0\})\), 
then, for every \(\varphi, \psi \in C_c^\infty(\R^d)\) that equal \(1\) in a neighborhood of \(0\), both distributions 
\begin{equation}
\label{eq892}
\Fourier{(\varphi T)} - (1-\psi) \Fourier{T} 
\quad \text{and} \quad \Fourier{((1-\varphi)T)} - \psi \Fourier{T}
\end{equation}
are represented in \(\R^d\) by Schwartz functions.
\end{lemma}

\resetconstant
\begin{proof}[Proof of Lemma~\ref{fouriertransform}]
Denote the distributions in \eqref{eq892} by \(L_1\) and \(L_2\), respectively. 
Observe that \(L_1 + L_2 = 0\).
Since \(\varphi T\) is a distribution with compact support, \(\Fourier{(\varphi T)}\) can be represented by a smooth function in \(\R^d\); see \cite{Grafakos}*{Theorem~2.3.21}.
Moreover, by homogeneity of \(f\), it follows from Theorem~\ref{HormanderFourier} that \(\Fourier{T}\) satisfies 
\begin{equation*}
\langle \Fourier{T}, \phi \rangle
= \int_{\R^d} h \phi 
\quad \text{for every \(\phi \in C_c^\infty(\R^d \setminus \{0\})\),}
\end{equation*}
for some smooth function \(h \colon  \R^d \setminus \{0\} \to \CC\).
Since \(1 - \varphi = 0\) in a neighborhood of \(0\), we deduce that \((1-\varphi) \Fourier{T}\) is  represented in \(\R^d\) by the smooth function \((1 - \psi)h\).
Hence, by linearity, \(L_1\) is represented in \(\R^d\) by a smooth function \(\DistribF{L_1}\).
As \(L_2 = -L_1\), the same holds for the distribution \(L_2\) and for the sequel we write in \(\R^d \setminus \{0\}\),
\[
\DistribF{L_2} = 
\DistribF{\Fourier{((1-\varphi)T)}} - \psi \DistribF{\Fourier{T}}.
\]

Since the function \(\psi \FourierF{T}\) is smooth in \(\R^d \setminus \{0\}\) and has compact support in \(\R^d\), to show that \(\DistribF{L_2}\) is a Schwartz function, it suffices to show that 
\[
p \vcentcolon= \FourierF{((1-\varphi)T)}
= \FourierF{((1-\varphi)h)}
\] 
and its derivatives have a fast decay at infinity.
More precisely, for every multi-index \(\nu\) and every \(j \in \N\) sufficiently large, we have
\begin{equation}\label{dacayofp}
|\partial^{\nu} p(\xi)|
\le \Cl{cte919} |\xi|^{-2j}
\quad \text{for every \(|\xi| \ge 1\),}
\end{equation}
where \(\Cr{cte919} > 0\) is some constant depending on \(\nu\) and \(j\).
Since \((1-\varphi)T\) is represented in \(\R^d\) by the smooth function \((1-\varphi)h\), for any $j\in \mathbb{N}$ we have
\[
(2\pi |\xi|)^{2j}\partial^{\nu}p(\xi)=\FourierF{P_{j, \nu}},
\]
where $P_{j, \nu}(x) \vcentcolon=\Delta^j\bigl((-2\pi i x)^{\nu} (1-\varphi)h\bigr)$ and \(\Delta\) is denotes the Laplacian.
Thanks to the Leibniz rule, the fact that $\varphi$ has compact support and the homogeneity of $h$, we deduce that
\[
\int_{\mathbb{R}^d} |P_{j, \nu}| 
\le \C \int_{\partial B_1} h \dif\sigma \int_1^{\infty}\rho^{\alpha-1-|\nu|} \dif\rho,
\]
where $\alpha \in \mathbb{R}$ is the degree of homogeneity of $h$.
When $2j>d- \alpha + |\nu|$ we obtain $P_{j, \nu}\in L^1(\mathbb{R}^d)$ and its Fourier transform is a bounded continuous function. 
Then, given a multi-index \(\nu\), taking any $j$ as above we deduce that $p$ satisfies \eqref{dacayofp}.
It follows that \(\DistribF{L_2}\) is a Schwartz function, which completes the proof.
\end{proof}

\begin{proof}[Proof of Proposition~\ref{propositionFourierNonHomogeneous}]
	We may assume for simplicity that \(a = b = 1\), \(r = 1\) and \(R = 2\).
	Given a radial function  $\varphi_{1}\in C^{\infty}_c(\mathbb{R}^d)$ equal to $1$ in the ball $B_{1}(0)$ and $0$ in the complement of the ball $B_{2}(0)$, we decompose \(g\) as follows
	\begin{equation} \label{gdecomposition}
	  g = g_{1} + g_{2} + g_{3},  
	\end{equation}
	where
\[
g_1 
\vcentcolon= \varphi_1 g 
= \varphi_1 f_{\alpha} 
\quad \text{and} \quad 
g_2 
\vcentcolon={} (1-\varphi_2) g
=(1-\varphi_2) f_{\beta}
\]
and \(f_{t}(x) = 1/\abs{x}^{d - t}\) with \(t = \alpha, \beta\).{}
Note that \(g_3 = g-g_1-g_2\) is smooth, vanishes in a neighborhood of \(0\) and has compact support.
We may thus associate to \(g_{3}\) a distribution \(\DistribS{g_{3}}\) defined by integration with respect to \(g_{3}\).{}
Then, by definition \(\DistribS{g_{3}}\) is represented by \(g_{3}\) in \(\R^{d}\) and its Fourier transform \(\Fourier{\DistribS{g_{3}}}\) is represented in \(\R^{d}\) by
\[{}
\FourierF{\DistribS{g_{3}}} = \widehat{g_{3}}\,.
\]
Since \(0 < \alpha < d\), by Example~\ref{exampleHomogeneous} the tempered distribution \(T_{\alpha}\) is represented by \(f_{\alpha}\) in \(\R^{d}\).{}
We then have that the tempered distribution \(\DistribS{g_{1}} \vcentcolon= \varphi_{1}T_{\alpha}\) is represented by \(g_{1} = \varphi_{1} f_{\alpha}\) in \(\R^{d}\).{}
We then deduce from Lemma~\ref{fouriertransform} applied with \(\varphi = \psi = \varphi_{1}\) that there exists a Schwartz function \(\eta_{1}\) representing a distribution \(\DistribS{\eta_{1}}\) such that
\[{}
\Fourier{\DistribS{g_{1}}}
= (1 - \varphi_{1}) \Fourier{T_{\alpha}} + \DistribS{\eta_{1}}.
\]
Since \(\Fourier{T_{\alpha}}\) is represented in \(\R^{d}\) by \(c_{\alpha}/\abs{\xi}^{\alpha}\), we deduce that \(\Fourier{\DistribS{g_{1}}}\) is represented in \(\R^{d}\) by
\begin{equation}\label{helphelp}
\FourierF{\DistribS{g_{1}}}
= (1 - \varphi_{1}) \frac{c_{\alpha}}{\abs{\xi}^{\alpha}} + \eta_{1}.
\end{equation}
Finally, since \(\beta < d\) and \(\beta \ne -2\ell\) for any \(\ell \in \N\), by Examples~\ref{exampleHomogeneous}, \ref{exampleFourierHomogeneousNegative} and~\ref{exampleFourier-2k} the tempered distribution \(T_{\beta}\) is represented by \(f_{\beta}\) in \(\R^{d} \setminus \{0\}\).{}
Since \(1 - \varphi_{1}\) vanishes in a neighborhood of \(0\), we then have that the tempered distribution \(\DistribS{g_{2}} \vcentcolon= (1 - \varphi_{1})T_{\beta}\) is represented by \(g_{2} = (1 - \varphi_{1}) f_{\beta}\) in \(\R^{d}\).{}
Moreover, from Lemma~\ref{fouriertransform} there exists a Schwartz function \(\eta_{2}\) representing a distribution \(\DistribS{\eta_{2}}\) such that
\[{}
\Fourier{\DistribS{g_{2}}}
= \varphi_{1} \Fourier{T_{\beta}} + \DistribS{\eta_{2}}\,.
\]
Observe that \(\Fourier{T_{\beta}}\) is represented in \(\R^{d}\) by \(c_{\beta}/\abs{\xi}^{\beta}\).
We deduce that \(\Fourier{\DistribS{g_{2}}}\) is represented in \(\R^{d}\) by
\[{}
\FourierF{\DistribS{g_{2}}}
= \varphi_{1} \frac{c_{\beta}}{\abs{\xi}^{\beta}} + \eta_{2}\,.
\]
We conclude that the tempered distribution \(S_{g} \vcentcolon= S_{g_{1}} + S_{g_{2}} + S_{g_{3}}\) is represented by \(g\) in \(\R^{d}\) and its Fourier transform is represented in \(\R^{d}\) by a function of the form in the statement.
\end{proof}

Let us now consider the case $\beta$ equal to a negative even integer.

\begin{proposition}
\label{propositionFourierNonHomogeneousb2k}
Take \(0 < \alpha < d\) and  $\beta =-2\ell$ for some \(\ell \in \N\).
Let \(g \in C^{\infty}(\R^d \setminus \{0\})\) be a radial function such that
\[
g(x) = 
\begin{cases}
{a}/{|x|^{d - \alpha}}
& \text{if\/ \(|x| \le r\),}\\
{b}/{|x|^{d - \beta}}
& \text{if\/ \(|x| \ge R\),}
\end{cases}
\]
where \(0 < r < R\) and \(a, b > 0\).
Then, there exists a tempered distribution \(\DistribS{g}\) in \(\R^{d}\) that is represented in \(\R^{d}\) by \(g\) and its Fourier transform \(\Fourier{\DistribS{g}}\) is represented in \(\R^{d}\) by a radial function \(\FourierF{\DistribS{g}} \in C^{\infty}(\R^d \setminus \{0\})\) that satisfies
\[
\FourierF{\DistribS{g}}(\xi)
= 
\begin{cases}
(-A_{2\ell}\log|\xi|+\lambda_{2\ell})|\xi|^{2\ell}
+\eta(\xi) 
& \text{for \(|\xi|\le r\),}\\
ac_{\alpha}/|\xi|^{\alpha} + \zeta(\xi) 
& \text{for \(|\xi|\ge R\),}
 \end{cases}
\]
where \(\eta\) and \(\zeta\) are Schwartz functions.
\end{proposition}

\begin{proof}
Following the same notation of the proof of Proposition~\ref{propositionFourierNonHomogeneous}, let us decompose $g$ as in \eqref{gdecomposition}. The treatment of $g_1$ and $g_3$ is the same as before.
To deal with $g_2$ in the case \(\beta = -2\ell\), for some \(\ell \in \N\), notice that by Examples~\ref{exampleFourier-2k}, \(T_{\beta}\) is again represented by \(f_{\beta}\) in \(\R^{d} \setminus \{0\}\).{} Therefore, the tempered distribution \(\DistribS{g_{2}} \vcentcolon= (1 - \varphi_{1})T_{\beta}\) is represented by \(g_{2} = (1 - \varphi_{1}) f_{\beta}\) in \(\R^{d}\)
and, again by Lemma~\ref{fouriertransform}, there exists a Schwartz function \(\eta_{2}\) representing a distribution \(\DistribS{\eta_{2}}\) such that
\[{}
\Fourier{\DistribS{g_{2}}}
= \varphi_{1} \Fourier{T_{\beta}} + \DistribS{\eta_{2}}\,.
\]
Since in this case \(\Fourier{T_{\beta}}\) is represented in \(\R^{d}\) by \((-A_{\ell}\log|\xi|+\lambda_\ell)|\xi|^{2\ell}\), we deduce that \(\Fourier{\DistribS{g_{2}}}\) is represented in \(\R^{d}\) by
\[{}
\FourierF{\DistribS{g_{2}}}
= \varphi_{1} (-A_{\ell}\log|\xi|+\lambda_\ell)|\xi|^{2\ell} + \eta_{2}\,.
\qedhere
\]
\end{proof}

\section{Existence of the representation formula}
\label{sectionexistenceofthekernel}

We prove in this section the representation formula in Theorem~\ref{introrepfor} and then apply it to establish the nonlocal Sobolev inequality associated to the sum of Lebesgue spaces (Theorems~\ref{introsobemb1} and~\ref{introsobemb2}).
We rely on the next proposition with \(\gamma = 2\), which is possible as we assume that \(d \ge 3\).

\begin{proposition}
\label{propositionExistenceVFirst}
Take \(0 < \alpha< \gamma < d\) and \(\beta < \min{\{\gamma, 1\}}\) with $\beta\neq 0$. 
If \(g \in C^\infty( \R^d \setminus \{0\})\) is a positive radial function  that satisfies \eqref{assurepositivityofFourier} and such that
\[
g(x) = 
\begin{cases}
{a}/{|x|^{d - \alpha}}
& \text{if\/ \(|x| \le r\),}\\
{b}/{|x|^{d - \beta}}
& \text{if\/ \(|x| \ge R\),}
\end{cases}
\]
where \(0 < r < R\) and \(a > 0, b\ge0\),
then there exists a radial function \(\omega \in C^\infty( \R^d \setminus \{0\} )\) such that, for every \(x \in \R^d \setminus \{0\}\), we have \(\abs{\omega} * g(x) < \infty\) and
\begin{equation}\label{convolutionequation}
\omega * g (x)
= \frac{1}{|x|^{d-\gamma}}\,.
\end{equation}
Moreover, if $b>0$, for every multi-index \(\nu\), 
\begin{equation}\label{omegadecay}
\abs{\partial^{\nu}\omega(x)} 
\le 
\begin{cases}
C/|x|^{d - (\gamma - \alpha) + |\nu|} 
& \text{if \(|x| \le 1\),}\\
C/|x|^{d - (\gamma - \beta^+) + |\nu|}
& \text{if \(|x| \ge 1\),}
\end{cases}
\end{equation}
where $\beta_+=\max\{\beta,0\}$, while, if $b=0$,
\begin{equation}\label{omegadecayb=0}
\abs{\partial^{\nu}\omega(x)} 
\le 
\begin{cases}
C/|x|^{d - (\gamma - \alpha) + |\nu|} 
& \text{if \(|x| \le 1\),}\\
C/|x|^{d - \gamma  + |\nu|}
& \text{if \(|x| \ge 1\).}
\end{cases}
\end{equation}
\end{proposition}

Before proving Proposition~\ref{propositionExistenceVFirst}, we need some technical lemmas.

\begin{lemma}\label{helpinglemma}
Let $0 < \delta < d$ and $\phi \in C^{\infty}_c(\mathbb{R}^d)$ be a radial function. 
If \(S \colon  \R^d \setminus \{0\} \to \R\) is the function defined by
\[
S(\xi) = \frac{\phi(\xi)}{|\xi|^{\delta}},
\]
then, for every multi-index \(\nu\) and every \(x \in \R^d\),
\[
|\partial^{\nu}\widecheck{S}(x)|
\le\frac{C'}{(1+|x|)^{d -\delta +|\nu|}}\,.
\]
\end{lemma}

\resetconstant
\begin{proof}[Proof of Lemma~\ref{helpinglemma}]
We write \(\phi\) for every \(\xi \in \R^d\) as $\phi(\xi)=h(|\xi|)\psi(\xi)$, where $h\in C^{\infty}(\mathbb{R})$ is even and $\psi \in C_c^\infty(\R^d)$ is a nonnegative radial function equal to one in $B_r$ and supported in $B_{2r}$, for some $r>0$ large enough. 
By symmetry of \(h\), the $2$nd order Taylor polynomial of $h$ at zero has the form \( \lambda_0 + \lambda_1 t^{2}\).
We then write \(S\) as
\[
S(\xi)= P(\xi) + R(\xi),
\]
where
\[
P(\xi) \vcentcolon= \frac{\lambda_0 + \lambda_1 |\xi|^{2}}{|\xi|^{\delta}}\psi(\xi)
\quad \text{and} \quad
R(\xi)\vcentcolon= \frac{h(|\xi|)-(\lambda_0 + \lambda_1 |\xi|^{2})}{|\xi|^{\delta}}\psi(\xi).
\]
By Example~\ref{exampleHomogeneous} and Lemma~\ref{fouriertransform} applied to both homogeneous terms of $P$, we conclude that
\begin{equation}
\label{eq-1142}
|\partial^{\nu}\widecheck{P}(x)|
\le\frac{\C}{(1+|x|)^{d  -\delta +|\nu|}}\,,
\end{equation}
for every \(x \in \R^m\) and every multi-index \(\nu\).
Since \(h\) is smooth and even, for any multi-indexes $\mu$ and $\nu$ the remainder \(R\) satisfies
\[
\xi^{\nu}\partial^{\mu}R(\xi)
= O(|\xi|^{-\delta + 4 + |\nu| - |\mu|}) \quad \text{as \(\xi\to 0\).}
\]
As $R$ has compact support, we deduce that for every multi-index \(\mu\) such that \(|\mu|< d - \delta + 4 + |\nu|\), we have 
\[
\int_{\R^d} |\xi|^{|\nu|}|\partial^{\mu}R| \dif \xi < \infty.
\]
This in turn implies that \(\widecheck{R}\) is smooth and
\begin{equation}
\label{eq-1155}
|\partial^{\nu}\widecheck{R}(x)|\le \frac{\C}{(1 + |x|)^{m}}\,,
\end{equation}
for every \(x \in \R^d\) and every nonnegative integer $m< d - \delta +4 + |\nu|$.
We may take in particular \(m = d  -\delta +|\nu|\), and the conclusion then follows from the combination of \eqref{eq-1142} and \eqref{eq-1155}.
\end{proof}

\begin{lemma}\label{Fcheck} 
Let $0 < \delta < d$, $\theta>0$, $a, b \in C^{\infty}(\R^d)$ be two positive radial functions and $\varphi \in C_{c}^{\infty}(\R^d)$ be a radial function.
If \(H \colon  \R^d \setminus \{0\} \to \R\) is the function
\begin{equation*}
   H(\xi)\vcentcolon=
\dfrac{\varphi(\xi)}{|\xi|^{\delta}(a(\xi) + b(\xi) |\xi|^{\theta})}\,,
\end{equation*}
then, for every multi-index \(\nu\),
\[
|\partial^{\nu}\widecheck{H}(x)|
\le\frac{C}{(1+|x|)^{d -\delta +|\nu|}}\,.
\]
\end{lemma}

\resetconstant
\begin{proof}[Proof of Lemma~\ref{Fcheck}]
 Let $k$ be the smallest integer such that \(k > 1/\theta\) and consider the following finite sum 
\begin{equation}\label{sumj}
H_k(\xi)=\frac{1}{ a(\xi)|\xi|^{\delta}}\sum_{j=1}^{k} \Bigl(-\frac{b(\xi)}{a(\xi)} \Bigr)^{j-1}|\xi|^{(j-1)\theta}\varphi(\xi).
\end{equation}
The generic term of \eqref{sumj} is in the form
\[
S_j(\xi)=\phi(\xi)|\xi|^{(j-1)\theta-\delta},
\]
with $\phi$ a smooth radially symmetric function with compact support. Then we can apply Lemma~\ref{helpinglemma} to deduce that,
for every multi-index \(\nu\),
\[
|\partial^{\nu}\widecheck{S}_j(x)|
\le\frac{\C}{(1+|x|)^{d +(j-1)\theta-\delta +|\nu|}}\,.
\]
The leading term is the one corresponding to $j=1$.
Looking at the definition \eqref{sumj} and using the linearity of the Fourier transform we conclude that
\begin{equation}\label{Scheck}
|\partial^{\nu}\widecheck{H}_k(x)|
\le\frac{\C}{(1+|x|)^{d -\delta +|\nu|}}\,.
\end{equation}

\begin{Claim}
The reminder $B_k \vcentcolon= H -H_k$ satisfies
\begin{equation}\label{Bdefinition}
B_k(\xi)=\Bigl(-\frac{b(\xi)}{a(\xi)}\Bigr)^{k}|\xi|^{k\theta-\delta}\frac{\varphi(\xi)}{a(\xi)+b(\xi)|\xi|^{\theta}} \quad \mbox{for \(k=1, 2, \ldots\)}
\end{equation}
\end{Claim}

\begin{proof}[Proof of the Claim]
We proceed by induction. 
For $k=1$ we have that
\[
\begin{split}
B_1(\xi)
& = \dfrac{\varphi(\xi)}{ |\xi|^{\delta}}\dfrac{1}{a(\xi) + b(\xi) |\xi|^{\theta}}-\frac{\varphi(\xi)}{ a(\xi)|\xi|^{\delta}}\\
& = \frac{\varphi(\xi)}{ |\xi|^{\delta}}\Big(\dfrac{1}{a(\xi) + b(\xi) |\xi|^{\theta}}-\frac{1}{a(\xi)}\Big)=-\frac{b(\xi)}{a(\xi)}|\xi|^{\theta-\delta}\frac{\varphi(\xi)}{a(\xi)+b(\xi)|\xi|^{\theta}}\,.
\end{split}
\]
Take now $k\ge1$ and assume that \eqref{Bdefinition} holds. 
Then,
\[
\begin{split}
B_{k+1}(\xi)& =H-H_k-\frac{1}{ a(\xi)} \Bigl(-\frac{b(\xi)}{a(\xi)} \Bigr)^{k}|\xi|^{k\theta-\delta}\varphi(\xi)\\
& =\Bigl(-\frac{b(\xi)}{a(\xi)}\Bigr)^{k}|\xi|^{k\theta-\delta}\frac{\varphi(\xi)}{a(\xi)+b(\xi)|\xi|^{\theta}}-\frac{1}{ a(\xi)} \Bigl(-\frac{b(\xi)}{a(\xi)} \Bigr)^{k}|\xi|^{k\theta-\delta}\varphi(\xi)\\
& =\Bigl(-\frac{b(\xi)}{a(\xi)}\Bigr)^{k}|\xi|^{k\theta-\delta}\varphi(\xi)\Big(\frac{1}{a(\xi) + b (\xi)|\xi|^{\theta}}-\frac{1}{a(\xi)}\Big)\\
& =\Bigl(-\frac{b(\xi)}{a(\xi)}\Bigr)^{k+1}|\xi|^{(k+1)\theta-\delta}\frac{\varphi(\xi)}{a(\xi) + b(\xi) |\xi|^{\theta}}\,,
\end{split}
\]
and the claim is proved.
\end{proof}

The fact that \(B_k\) is integrable with compact support implies that $\widecheck{B}_k\in C^{\infty}(\mathbb{R}^d)$. To estimate the decay of $\widecheck{B}_k(x)$ as $x\to\infty$, we may rely on the integrability of derivatives of $B_k$. 
For any multi-indices $\mu$ and $\nu$, it follows that
\[
\xi^{\nu}\partial^{\mu}B_k(\xi)=O(|\xi|^{k\theta-\delta+|\nu|-|\mu|}) \quad \mbox{for } |\xi|\le 1.
\]
Using again that $B_k$ has compact support, we deduce that, if \(|\mu|<k\theta+d-\delta+|\nu|\), then $\xi^{\nu}\partial^{\mu}B_k\in L^1(\mathbb{R}^d)$. This in turn implies that
\[
|\partial^{\nu}\widecheck{B}_k(x)|\le \frac{\C}{|x|^{m}}\,,
\]
for any non-negative integer $m<k\theta-\delta+|\nu|+d$.
Since $k\theta\ge 1$, we can take $m=d-[\delta]+|\nu|$ and, recalling that $\widecheck{B}_k$ is smooth, we conclude that
\begin{equation}\label{bcheck}
\abs{\partial^{\nu}\widecheck{B}_k(x)}\le \frac{\C}{(1+|x|)^{d-[\delta]+|\nu|}}  \quad \mbox{in  } \mathbb{R}^d.
\end{equation}
Thanks to \eqref{Scheck}, \eqref{bcheck} and the fact that $H=H_k+B_k$\,, it follows that
\[
\abs{\partial^{\nu}\widecheck{H}(x)}
\le \frac{\C}{(1+|x|)^{d-\delta+|\nu|}}
 \quad \text{in \(\mathbb{R}^d\).}
 \qedhere
\]
\end{proof}

\resetconstant
\begin{proof}[Proof of Proposition~\ref{propositionExistenceVFirst}]
Assume for simplicity that \(a = b = 1\), \(r = 1\) and \(R = 2\), and let us start dealing with the case where $\beta \not\in -2\N$.
By Proposition~\ref{propositionFourierNonHomogeneous}, there exists a tempered distribution \(\DistribS{g}\) that is represented by \(g\) in \(\R^{d}\) such that its Fourier transform \(\Fourier{\DistribS{g}}\) in \(\R^{d}\) by 
\[
\FourierF{\DistribS{g}}(\xi)
= 
\begin{cases}
c_{\beta}/|\xi|^{\beta}
+\zeta_1(\xi) 
& \text{for \(|\xi|\le 1\),}\\
c_{\alpha}/|\xi|^{\alpha} + \zeta_2(\xi) 
& \text{for \(|\xi|\ge 2\),}
 \end{cases}
\]
where \(\zeta_1\) and \(\zeta_2\) are Schwartz functions. Moreover, Proposition~\ref{propositionPositivity} assures us that $\FourierF{\DistribS{g}}(\xi)>0$ for all $\xi\in\mathbb{R}^d$.

Let us define $H$ as
\begin{equation}\label{Hdefinition}
H(\xi)\vcentcolon=\frac{c_{\gamma}}{|\xi|^{\gamma}\FourierF{\DistribS{g}}(\xi)}.
\end{equation}
Now we claim that $H$ is a function that represents a tempered distribution $\DistribS{H}$ in \(\R^{d}\) and that 
\[{}
\omega{}
\vcentcolon = \DistribF{\mathcal{F}^{-1} \DistribS{H}}
\]
satisfies \eqref{omegadecay}. 

As before, let us decompose $H$ as
\[
H_1 \vcentcolon= H\varphi_1,
\quad H_2 \vcentcolon=H(1-\varphi_2)
\quad \text{and} \quad H_3\vcentcolon=H-H_1-H_2\,.
\]
We handle each of these functions separately. 
Since $H_3$ is smooth with compact support, its Fourier transform is a Schwartz function.
To deal with $H_2$\,, we set
\begin{equation*}
D(\xi)
\vcentcolon=H_2(\xi)-\frac{1}{c_{\alpha}|\xi|^{\gamma-\alpha}}(1-\varphi_2(\xi))
=\frac{1}{c_{\alpha}|\xi|^{\gamma-2\alpha}}\frac{\zeta_2(\xi)}{c_{\alpha}+\zeta_2(\xi)|\xi|^{\alpha}}(1-\varphi_2(\xi)).
\end{equation*}
Thanks to the fast decay of $\zeta_2(\xi)$ at infinity and the fact that $D(\xi)$ is zero in a neighborhood of the origin, we deduce that $\xi^{\nu}\partial^{\mu}D(\xi)\in L^1(\mathbb{R}^d)$ for any multi-indexes $\mu$ and $\nu$. 
Therefore, $\widecheck{D}$ is a Schwartz function and we can apply Lemma~\ref{fouriertransform} to deduce that
\begin{equation*}
[\mathcal{F}^{-1}H_2](x)
=\frac{\C}{|x|^{d + \alpha - \gamma}}\varphi+\eta_1\,,
\end{equation*}
where \(\eta_1\) is a Schwartz function.
Let us focus now on the term $H_1$. 
If $0<\beta<d$, then \(H_1\) can be written as
\[
H_1(\xi)=\dfrac{\varphi_1(\xi)}{|\xi|^{\gamma-\beta}(c_{\beta} + \zeta_1 (\xi) |\xi|^{\beta})}\,,
\]
while, if  $\beta<0$ and different from an even integer, we have that
\[
H_1(\xi)=\dfrac{\varphi_1(\xi)}{|\xi|^{\gamma}(c_{\beta}|\xi|^{-\beta} + \zeta_1 (\xi) )}.
\]
In both cases we can apply Lemma~\ref{Fcheck} to deduce that for every multi-index \(\nu\),
\begin{equation}\label{Hfourier}
   |\partial^{\nu}\widecheck{H}_1(x)|
\le\frac{\C}{(1+|x|)^{d +\beta^+-\gamma +|\nu|}}. 
\end{equation}

Thanks to the linearity of the Fourier transform, the proof of the claim follows by summing up the information obtained for $H_1$, $H_2$ and $H_3$.
To prove that the function $\omega$ satisfies \eqref{convolutionequation}, notice that by construction it follows that
\[
H(\xi)\FourierF{\DistribS{g}}=\frac{c_{\gamma}}{|\xi|^{\gamma}}.
\]
Thanks to the properties of $H_1$, $H_2$, $H_3$, we have that
\[
T_{\gamma}=\mathcal{F}^{-1} Y_1+\mathcal{F}^{-1} Y_2+\mathcal{F}^{-1} Y_3\,,
\]
where $Y_i$ is defined as
\[
\langle Y_i, \eta \rangle 
\vcentcolon= \int_{\R^d}H_i\FourierF{\DistribS{g}}\eta \,.
\]
Recalling that $H_3$ is a Schwartz function we have that
\begin{equation}\label{qui}
\begin{split}
\langle \mathcal{F}^{-1} Y_3, {\eta} \rangle
& =\langle Y_3, \widecheck{\eta} \rangle
 =\int_{\R^d}\FourierF{\DistribS{g}}H_3 \widecheck{\eta} \\
& =\langle \Fourier{\DistribS{g}}, H_3 \widecheck{\eta} \rangle=\langle \DistribS{g}, \widehat{H_3} * \eta \rangle =\int_{\R^d}g \widehat{H_3} * \eta=\int_{\R^d} \widecheck{H_3}* g  \eta\,.
\end{split}
\end{equation}
To handle the first piece, notice that
\[
\langle \mathcal{F}^{-1} Y_1, {\eta} \rangle=\langle Y_1, \widecheck{\eta} \rangle=\int_{\R^d}\FourierF{\DistribS{g}}H_1 \widecheck{\eta} =\lim_{\epsilon\to 0} \int_{\R^d}\FourierF{\DistribS{g}} (H_1\widecheck{\eta})*\rho_{\epsilon}\,,
\]
where the limit is justified by the fact that convolution in the integral converges in $L^q(\mathbb{R}^d)$ for all $q<\frac{d}{\gamma-\beta^+}$, that $\FourierF{\DistribS{g}}\in L^p\loc(\mathbb{R}^n)$ for any $p<\frac{d}{\beta^+}$, that $q'<\frac{d}{\beta^+}$, and that $H_1$ has compact support.
Moreover, we have
\[
\int_{\R^d}\FourierF{\DistribS{g}} (H_1\widecheck{\eta})*\rho_{\epsilon}=\langle \DistribS{g}, [(H_1\widecheck{\eta})*\rho_{\epsilon}]^{\wedge} \rangle=\int_{\R^d}g \widehat{H_1}*\eta \, \widehat{\rho_{\epsilon}}\,.
\]
Therefore, 
\begin{equation}\label{quo}
\langle \mathcal{F}^{-1} Y_1, {\eta} \rangle=\int_{\R^d}\widecheck{H_1}*g \eta.
\end{equation}
Finally, we have that 
\begin{equation}\label{qua}
\begin{split}
\langle \mathcal{F}^{-1} Y_2, {\eta} \rangle & = \int_{\R^d} \FourierF{\DistribS{g}} H_2\widecheck{\eta} = \langle \DistribS{g}, (H_2 \widecheck{\eta})^{\wedge}  \rangle\\
& =\int_{\R^d} g (H_2 \widecheck{\eta})^{\wedge}
= \int_{\R^d} g \FourierF{H_2}*\eta=\int_{\R^d} [\mathcal{F}^{-1}H_2]* g \eta \,,
\end{split}
\end{equation}
where we have used that $H_2\widecheck{\eta}$ belongs to the Schwartz class, that $(H_2 \widecheck{\eta})^{\wedge}=\FourierF{H_2}*\eta$ and $\FourierF{H_2}(-\xi)=[\mathcal{F}^{-1}(H_2)](\xi)$. 
Putting together \eqref{qui}, \eqref{quo} and \eqref{qua}, we conclude that the previously defined $\omega$ satisfies \eqref{convolutionequation}. 

Let us address now the choice $b=0$ in the definition of $g$. Using the same notation as in the proof of Proposition~\ref{propositionFourierNonHomogeneous}, we see that $g = g_1$, where $g_1$ has been defined in \eqref{gdecomposition}. 
Therefore, recalling \eqref{helphelp}, it follows that
\[
\FourierF{\DistribS{g}}=\FourierF{\DistribS{g_{1}}}
= (1 - \varphi_{1}) \frac{c_{\alpha}}{\abs{\xi}^{\alpha}} + \zeta_{1}\,.
\]
for a suitable radial Schwartz function $\zeta_{1}$. Now we proceed as before, defining $H$ as in \eqref{Hdefinition} and decomposing it in $H_1$, $H_2$, and $H_3$. Recalling that $\FourierF{\DistribS{g}}>0$, it follows that 
\[
H_1(\xi)=\dfrac{\varphi_1(\xi)}{|\xi|^{\gamma} \zeta_1 (\xi) }=\dfrac{\tilde \varphi(\xi)}{|\xi|^{\gamma} }\,,
\]
where $\tilde \varphi=\varphi_1/\zeta_1$ is smooth, radial, and with compact support. Therefore \eqref{Hfourier} holds true and the rest of the proof follows as for $b>0$ and $\beta<0$ different from a negative even number.

Finally let us go back to the case $\beta=-2\ell$ for some $\ell=1,2,\ldots$
Recalling \eqref{eq921}, the expression for $\FourierF{\DistribS{g}}(\xi)$ becomes
\[
\FourierF{\DistribS{g}}(\xi)
= 
\begin{cases}
(-A_k\log{|\xi|}+\lambda_k)|\xi|^{k}
+\eta(\xi) 
& \text{for \(|\xi|\le 1\),}\\
c_{\alpha}/|\xi|^{\alpha} + \eta(\xi) 
& \text{for \(|\xi|\ge 2\).}
 \end{cases}
\]
The strategy of the proof is the same as in the previous case, we keep the same notation, and only highlight the differences. In particular we have that $H_1(\xi)$ modifies as follows
\[
H_1(\xi)=\dfrac{\varphi_1(\xi)}{ |\xi|^{\gamma}}\dfrac{1}{(-A_k\log{|\xi|}+\lambda_k)|\xi|^{k}+\eta(\xi)}\,.
\]
We further decompose $H_1$ as follows
\[
H_1(\xi)=S(\xi)+B(\xi),
\]
where
\begin{equation*}
S(\xi) \vcentcolon=\frac{1}{ \eta(\xi)|\xi|^{\gamma}}\varphi_1(\xi)
\end{equation*}
and
\begin{equation*}
B(\xi)=-\frac{(-A_{k}\log{|\xi|}+\lambda_k)|\xi|^{k-\gamma}}{\eta(\xi)\big[\eta(\xi) + (-A_{k}\log{|\xi|}+\lambda_k)|\xi|^{k}\big]}\varphi_1(\xi).
\end{equation*}
Using Lemma~\ref{helpinglemma}, we deduce that
\[
\abs{\partial^{\nu}\widecheck{S}(x)}
\le \frac{\C}{(1+|x|)^{d-\gamma+|\nu|}}
 \quad \mbox{in  } \mathbb{R}^d.
\]
To deal with $\widecheck{B}$, notice that, for any multi-indices $\mu,\nu$, it follows that
\[
|\xi^{\nu}\partial^{\mu}B(\xi)|\le \C|\log{|\xi|}| \, |\xi|^{k-\gamma+|\nu|-|\mu|} \quad \mbox{for } |\xi|\le1/2.
\]
Together with the fact that $B$ has compact support, this implies that, for every nonnegative integer \(m<k-\gamma+|\nu|+d\),
\[
|\partial^{\nu}\widecheck{B}(x)|\le \frac{\C}{|x|^{m}}.
\]
Choosing $m=d- \lfloor \gamma \rfloor +|\nu|$ and using the linearity of the Fourier transform,
we conclude that again that, for any multi-index $\nu$,
\[
\abs{\partial^{\nu}\widecheck{H}_1(x)}
\le \frac{\C}{(1+|x|)^{d-\gamma+|\nu|}}
 \quad \mbox{in  } \mathbb{R}^d.
\]
The remaining of the proof follows along the same lines.
\end{proof}

\begin{proof}[Proof of Theorem~\ref{introrepfor}] 
Consider at first the case \eqref{eqGLocal} and let us apply Proposition~\ref{propositionExistenceVFirst} with $\gamma=2$, $\alpha=1-s$, $a=1$ and $b=0$. We deduce that there exists $\omega\in C^{\infty}(\R^d\setminus\{0\})$ such that $\omega*g(x)=1/|x|^{d-2}$. Estimate \eqref{omegadecayb=0} implies that the vector valued function $V=-\frac1{(N-2)\sigma_d}\nabla \omega$ belongs to $L^1+L^{d}$. Moreover, a straightforward computation shows that
\[
V*g(x)=\frac{1}{\sigma_d}\frac{x}{|x|^d}\,.
\]
Therefore, Proposition~\ref{corollaryfromVtorepresentationfomula} implies the desired result. To deal with the remaining cases \eqref{eqGIntermediate} and \eqref{eqGTwo}, we have just to apply Proposition~\ref{propositionExistenceVFirst} with the appropriate values of the considered parameters, the rest of the proof being the same.    
\end{proof}

\resetconstant
\begin{proof}[Proof of Theorem~\ref{introsobemb1}]
 Theorem~\ref{introrepfor} assures us that
 \[
 u=V*\mathcal{G}u.
 \]
 Moreover, under the considered assumptions, Proposition~\ref{propositionExistenceVFirst} implies that $V=c\nabla \omega$ satisfies estimate \eqref{introV1}. Therefore, decomposing \(V\) as 
 \[
 V = V\varphi_1+V(1-\varphi_1) =\vcentcolon V_1+V_2 \,,
 \]
we obtain that
\[
\abs{V_1(z)} + \abs{z} \abs{\nabla V_1(z)} 
\le \frac{\C}{|z|^{d - s}} \quad \text{and} \quad \abs{V_2(z)} + \abs{z} \abs{\nabla V_2(z)} 
\le \frac{\C}{|z|^{d - 1}}\,.
\]
When $p=1$, Proposition~\ref{propositionFractionalSobolev} implies that
\[
\|V_1*\mathcal{G}u\|_{L^{\frac{d}{d-s}}}\le \C\|\mathcal{G}u\|_{L^1} 
\quad \text{and} \quad  
\|V_2*\mathcal{G}u\|_{L^{\frac{d}{d-1}}}\le \C\|\mathcal{G}u\|_{L^1}\,.
\]
Thanks to the linearity of the convolution, we obtain the desired result.

To deal with the case $1 < p < d/s$, it is enough to notice that the estimates
\[
\|V_1*\mathcal{G}u\|_{L^{\frac{pd}{d-sp}}}\le \C\|\mathcal{G}u\|_{L^p} 
\quad \text{and} \quad 
\|V_2*\mathcal{G}u\|_{L^{\frac{pd}{d-p}}}\le \C\|\mathcal{G}u\|_{L^p}\,,
\]
follow by classical results on convolution operators, see for instance \cite{O}*{Theorem, p.\,139}. 
\end{proof}

The proof of Theorem~\ref{introsobemb2} follow the same structure as the proof of Theorem~\ref{introsobemb1}, now applying \eqref{introV2} rather than \eqref{introV1} and shall be omitted.

\section{Nonlocal Sobolev inequalities in Lorentz spaces}

Before proving Theorems~\ref{introsobemb3} and~\ref{introsobemb4}, let us first recall for the convenience of the reader the definition of the Lorentz spaces that are involved.
We need to introduce the decreasing rearrangement of a measurable function and some related notation, see \cite{Grafakos,O,oklander,Ziemer} for details.

\begin{definition}
    The \emph{decreasing rearrangement} of a measurable function $v\colon \R^d\to \R$ is defined for every \(0 < y < \infty\) as
    \[
    v^*(y)=\inf\{\tau\ge0 \ : A(\tau)<y\},
    \]
    where $A(t) \vcentcolon= |\{|v|>t\}|$.
\end{definition}

Taking
    \[
    v^{**}(y)=\frac1y\int_0^yv^*(\tau)\dif\tau,
    \]
the measurable function $v$ belongs to the Lorentz space $L^{p,q}(\R^d)$ whenever
\[
\|v\|_{L^{p,q}(\mathbb{R}^d)} \vcentcolon= \left(\int_0^{\infty}\big( \tau^{\frac{1}{p}}v^{**}(\tau)\big)^q\frac{\dif\tau}{\tau}\right)^{\frac1q}<\infty
\quad \text{for $1 < p< \infty$ and $1 \le q < \infty$},
\]
 and
\[
\|v\|_{L^{p,\infty}(\mathbb{R}^d)} \vcentcolon= \sup_{\tau >0}{\tau^{\frac 1 p}v^{**} (\tau)} < \infty
\quad \text{for $1 \le p \le \infty$ and \(q = \infty\).}
\]

We begin with the following estimate:

\begin{proposition}
\label{propositionFractionalSobolevp>1}
Let \(1 < p < \infty\), let \(0 < s < 1\) with $s p<d$, and let \(0 < t \le 1\) with $t p<d$.{}
Take $F\in L^p(\R^d,\R^d)$ and a measurable function \(v\colon \R^d\to \R^d \) such that, for almost every \(z \in \R^d\),
\begin{equation}\label{assonv}
\abs{v(z)} 
\le 
\begin{cases}
C/|z|^{d - s} 
& \text{if \(|z| \le 1\),}\\
C/|z|^{d - t }
& \text{if \(|z| > 1\).}
\end{cases}
\end{equation}
Then, the convolution $h\vcentcolon= v*F$ belongs to $(L^{\frac{pd}{d-s p},\,p}+L^{\frac{pd}{d-t p},\,p})(\R^d)$ and 
\[
\int_0^1 \tau^{\frac{d-s p}{d}} h^{**}(\tau)\frac{\dif\tau}{\tau}+\int_1^{\infty}\tau^{\frac{d-t p}{d}} h^{**}(\tau) \frac{\dif\tau}{\tau}
\le C' \|F\|_{L^p(\mathbb{R}^d)}^p.
\]
\end{proposition}
\resetconstant
\begin{proof}
Writing $v=v\chi_{B_1}+v\chi_{B_1^c}=v_1+v_2$, we can use assumption \eqref{assonv} to deduce that
\[
|v_1(z)|\le c/|z|^{d-s}  \quad \mbox{and} \quad |v_2(z)|\le c /|z|^{d-t}.
\]
Therefore, the embedding of the Riesz potential between Lorentz spaces, see e.g.~\cite[Theorem, p.~139]{O}, ensures that $h \vcentcolon= v_1*F+v_2*F$ is well defined and belongs to $ L^{\frac{pd}{d-s p},\,p}+L^{\frac{pd}{d-t p},\,p}$.
Let us now set, with an abuse of notation, $v^{*}=(|v|)^*$ and $v^{**}=(|v|)^{**}$. 
A direct computation shows that
\begin{equation}
\label{report}
v^{**}(\tau)\le \mathrm{v}(\tau) \vcentcolon= 
\begin{cases}
 c \tau^{-1 + s/d} 
& \text{if \(\tau \le 1\),}\\
 c \tau^{-1 + t/d }
& \text{if \(\tau > 1\),}
\end{cases}
\end{equation}
for a suitable constant $c > 0$. 
Thanks to the estimate for convolution operators provided in \cite[Lemma~1.6]{O}, we deduce that
\begin{equation}
\label{estimateRearrangement}
h^{**}(\tau)\le \int_\tau^{\infty}v^{**}(y)F^{**}(y) \dif y.
\end{equation}
Therefore, 
\[
\begin{split}
\int_{1}^{\infty} \tau^{\frac{d-t p}{d}} h^{**}(\tau)^p\frac{\dif\tau}{\tau}&\le 
\int_{1}^{\infty} \tau^{-\frac{t p}{d}} \left(\int_{\tau}^{\infty}v^{**}(y)F^{**}(y)\dif y\right)^p\dif\tau\\ 
&\le  c^p\int_{1}^{\infty}\tau^{-\frac{t p}{d}}\left(\int_{\tau}^{\infty}y^{-1+\frac{t}{d}}F^{**}(y)\dif y\right)^p\dif\tau
\le \C \|F^{**}\|_{L^p(\mathbb{R})}^p,
\end{split}
\]
where the second inequality follows from the estimates on $v^{**}$ and the last one by Hardy's inequality where we use that $t<d/p$\,; see \cite[Eq.~(9.9.10), p.~246]{HLP}.
On the other hand, integration by parts leads us to
\[
\begin{split}
B & \vcentcolon=\int_0^{1} \tau^{-\frac{s p}{d}} \left( \int_{\tau}^{\infty}v^{**}(y)F^{**}(y)\dif y\right)^p\dif\tau \\ & = \Cl{cte-1157} \left(\int_{1}^{\infty}v^{**}(y)F^{**}(y)\dif y\right)^p
+ \C\int_0^{1} \tau^{1-\frac{s p}{d}} \left(\int_\tau^{\infty}v^{**}(y)F^{**}(y)\dif y\right)^{p-1}v^{**}(\tau)F^{**}(\tau) \dif\tau\\
& \le \Cr{cte-1157} \left(\int_{1}^{\infty}v^{**}(y)^{p'}\dif y\right)^{p-1}\int_{1}^{\infty}F^{**}(y)^p \dif y + \C \int_0^{1} \tau^{1-\frac{s p}{d} }v^{**}(\tau)^pF^{**}(\tau)^p \dif\tau+ \frac B2,
\end{split}
\]
where we have used H\"older and Young inequalities in the third line. Thanks to the assumptions on $v$, we conclude that
\[
\int_{0}^{1} \tau^{\frac{d-s p}{d}} h^{**}(\tau)^p\frac{\dif\tau}{\tau}\le B\le \C \|F^{**}\|_{L^p(\mathbb{R})}^p.
\]
Combining both estimates for $h^{**}$ and using the fact that the norms \(\|F^{**}\|_{L^p(\mathbb{R})}\) and \(\|F\|_{L^p(\mathbb{R}^d)}\) are comparable, we obtain the desired result. 
\end{proof}

We now deduce the following estimates involving the truncations \(G_k\) and \(T_k\):

\begin{proposition} \label{truncationinlorentz}
    Let \(p\), \(s\) and \(t\) be as in the statement of Proposition~\ref{propositionFractionalSobolevp>1}.
    For every $F\in L^p(\R^d,\R^d)$ and every measurable function \(v\colon \R^d\to \R^d \) that satisfies \eqref{assonv}, the convolution $h= v*F$ satisfies
\[
\|G_{\bar k}(h)\|_{L^{\frac{pd}{d-ps},p}(\mathbb{R}^d)}+\|T_{\bar k}(h)\|_{L^{\frac{pd}{d-pt},p}(\mathbb{R}^d)} \le C \norm{F}_{L^{p}(\R^d)},
\]
with $\bar k \vcentcolon= \norm{F}_{L^{p}(\R^d)}$.
\end{proposition}

\resetconstant
\begin{proof}
Let us first notice that, using \eqref{estimateRearrangement},
\begin{equation}\label{estuno}
h^{**}(\tau)\le \int_{\tau}^{\infty}v^{**}(y)F^{**}(y)\dif y\le \left(\int_{\tau}^{\infty}\mathrm{v}(y)^{p'}\dif y\right)^{\frac1{p'}}\|F\|_{L^p}=\vcentcolon \mathrm{I}(\tau)\|F\|_{L^p}\,,
\end{equation}
where $\mathrm{v}$ is given in \eqref{report}.
By construction, the real function $\mathrm{I}\vcentcolon (0,\infty)\to(0,\infty)$ defined above is invertible.
Therefore,
\begin{equation}\label{estdue}
A(k)
=|\{|h|>k\}|
=|\{h^*>k\}|
\le \mathrm{I}^{-1}\bigl(k/\|F\|_{L^p}\bigr).
\end{equation}
 Set 
 \[
 \bar k=\|F\|_{L^p}, \quad h_1=G_{\bar k}(h) \quad \text{and} \quad h_2=T_{\bar k}(h).
 \]
 Since $h_1^*\le h^*\chi_{(0,A(\bar k))}$, it follows that
\[
\begin{split}
\|h_1\|_{L^{\frac{pd}{d-ps},p}}&\le \C \left(\int_0^{1}\tau^{\frac{d-s p}{d}} h^{**}(\tau)^p\frac{\dif\tau}{\tau}\right)^{\frac1p}+\C\left(\int_1^{\max\{A(\bar k),1\}}\tau^{\frac{d-s p}{d}} h^{**}(\tau)^p\frac{\dif\tau}{\tau}\right)^{\frac1p}\\
&\le \C \|F\|_{L^p}+h^{**}(1)(A(\bar k)-1)_+^{\frac1p} 
\le \C\|F\|_{L^p}\,,
\end{split}
\]
where we have used Proposition~\ref{propositionFractionalSobolevp>1} to estimate the first integral, and formula \eqref{estuno} with $\tau=1$ and \eqref{estdue} with $ k= \|F\|_{L^p(\mathbb{R}^d)}$ to estimate the second one.
To deal with $h_2$, notice that $h_2^*=T_{\bar k}(h^*)$. Therefore,
\begin{align*}
\|h_2\|_{L^{\frac{pd}{d-pt},p}}&\le  \left(\int_0^{1}\tau^{\frac{d-t p}{d}} \bar k^p\frac{\dif\tau}{\tau}\right)^{\frac1p}+\left(\int_{1}^{\infty}\tau^{\frac{d-t p}{d}} h^{**}(\tau)^p\frac{\dif\tau}{\tau}\right)^{\frac1p}\\
&\le \C \bar k +\left(\int_{1}^{\infty}\tau^{\frac{d-t p}{d}} h^{**}(\tau)^p\frac{\dif\tau}{\tau}\right)^{\frac1p}
\le \C\|F\|_{L^p}\,,
\end{align*}
where the last inequality follow by the choice of $\bar k$ and Proposition~\ref{propositionFractionalSobolevp>1}.
\end{proof}

\begin{proof}[Proof of Theorem~\ref{introsobemb3}] Under the considered assumptions, Proposition~\ref{propositionExistenceVFirst} implies that $V \vcentcolon=c \nabla \omega$ satisfies estimate \eqref{introV1} and that $u=V*\mathcal{G}u$ for all $u\in C^{\infty}_c(\mathbb{R}^d)$. 
Therefore, we apply Proposition~\ref{truncationinlorentz} to $h = u =V*\mathcal{G}u$, with $0 < s < 1$ and $t=1$, to get the desired result.
\end{proof}

The proof of Theorem~\ref{introsobemb4} follows along the lines of the proof of Theorem~\ref{introsobemb3}, by applying \eqref{introV2} rather than \eqref{introV1}.

\section*{Acknowledgments}

The authors thank D.~Spector for bringing to their attention the discussion of the Fourier transform in the book \cite{SKM} and the referee for helpful comments.
The first author (S.~Buccheri) received support from the Fonds de la Recherche scientifique (NFSR, grant CR 40006150), from the Austrian Science Fund (FWF, grant 10.55776/ESP9), and from the GNAMPA-INdAM Project 2025 ``Local and nonlocal equations with lower order terms" (CUP E5324001950001).
The second author (A.~C.~Ponce) was supported by the Fonds de la Recherche scientifique (F.R.S.--FNRS) under research grant J.0020.18.

\end{document}